\documentclass{amsart}

\usepackage{amssymb}

\theoremstyle{plain} 

\newtheorem{theorem}{Theorem}[section]
\newtheorem{lemma}[theorem]{Lemma}

\newtheorem{corollary}[theorem]{Corollary}
\newtheorem{conjecture}[theorem]{Conjecture}

\newtheorem{proposition}[theorem]{Proposition}

\newtheorem{hypothesis}[theorem]{Hypothesis}

\newcommand{\makeinvisible}[1]{}

\numberwithin{equation}{section}

\newcommand{\cc}{{\mathbb C}}
\newcommand{\pp}{{\mathbb P}}
\newcommand{\rr}{{\mathbb R}}

\newcommand{\zz}{{\mathbb Z}}

\newcommand{\aaa}{{\mathbb A}}

\newcommand{\Oo}{{\mathcal O}}

\begin{document}

\author[N. Mestrano]{Nicole Mestrano}
\address{CNRS, Laboratoire J. A. Dieudonn\'e, UMR 7351
\\ Universit\'e de Nice-Sophia Antipolis\\
06108 Nice, Cedex 2, France}
\email{nicole@math.unice.fr}
\urladdr{http://math.unice.fr/$\sim$nicole/} 

\author[C. Simpson]{Carlos Simpson}
\address{CNRS, Laboratoire J. A. Dieudonn\'e, UMR 7351
\\ Universit\'e de Nice-Sophia Antipolis\\
06108 Nice, Cedex 2, France}
\email{carlos@math.unice.fr}
\urladdr{http://math.unice.fr/$\sim$carlos/} 

\thanks{This research project was initiated on our visit to Japan supported by JSPS Grant-in-Aid for Scientific Research (S-19104002)}

%
%

\title[Bundles on a quintic surface]{Seminatural bundles of rank two, degree one and $c_2=10$ on a 
quintic surface}



\begin{abstract}
In this paper we continue our study of the moduli space of stable bundles of rank two
and degree $1$ on
a very general quintic surface. The goal in this paper is to understand the irreducible
components of the moduli  space in the first case in the ``good'' range, which is 
$c_2=10$. We show that there is a single irreducible component
of bundles which have seminatural cohomology, and conjecture that this is the only
component for all stable bundles. 
\end{abstract}

\dedicatory{Dedicated to the memory of Professor Masaki Maruyama}

\maketitle

This paper is the next in a series, starting with \cite{MestranoSimpson}, in which we study
the moduli spaces of rank two bundles of odd degree on a very general quintic hypersurface $X\subset \pp^3$.
This series is dedicated to Professor Maruyama, who brought us together in the study of moduli spaces,
a subject in which he was one of the first pioneers. 

In the first paper, we showed that the moduli space $M_X(2,1,c_2)$,
of stable bundles of rank $2$, degree $1$ and given $c_2$, is empty for $c_2\leq 3$, irreducible
for $4\leq c_2\leq 9$, and good (i.e.\ generically smooth of the expected dimension)
for $c_2\geq 10$. 
On the other hand, Nijsse has shown that the moduli space is
irreducible for $c_2\geq 15$ \cite{Nijsse} using the techniques of O'Grady
\cite{OGradyIrred} \cite{OGradyBasic}. This leaves open the question of irreducibility for $10\leq c_2\leq 14$.

\begin{conjecture}
\label{all}
The moduli space $M_X(2,1,10)$ is irreducible.
\end{conjecture}

We haven't yet formulated an opinion about the cases $11\leq c_2\leq 14$. 

In the present paper, due to lack of time and for length reasons, we treat a special case of the conjecture:
the case of bundles with {\em seminatural cohomology}, meaning that only at most one of $h^0(E(n))$,
$h^1(E(n))$ or $h^2(E(n))$ can be nonzero for each $n$. Let $M^{sn}_X(2,1,10)$ denote the 
open subvariety of the moduli space consisting of bundles with seminatural cohomology. 
In Section \ref{seminatural} we show that the seminatural condition is a consequence of
assuming just $h^0(E(1))=5$. The main result of this paper is:

\begin{theorem}
\label{main}
The moduli space $M^{sn}_X(2,1,10)$ is irreducible.
\end{theorem}

Recall from \cite{MestranoSimpson} that our inspiration to look at this question came from 
the recent results of Yoshioka, for the case of 
Calabi-Yau surfaces originating in \cite{Mukai}. Yoshioka
shows that the moduli spaces are irreducible for all positive values of $c_2$, when $X$ is an abelian or K3 surface
\cite{YoshiokaAbelian} \cite{YoshiokaK3}.
His results apply for example when $X$ is a general quartic hypersurface. We thought it was a natural
question to look at the case of a quintic hypersurface, which is
one of the first cases where $X$ has general type, with $K_X=\Oo _X(1)$ being as small as possible. 

{\em Remark on the difficulty of this project:}
We were somewhat surprised by the diversity of techniques needed to treat this question. Much of the difficulty stems
from the possibilities of overdetermined intersections which need to be ruled out at various places in the argument.
This question is inherently very delicate, because there is not, to our knowledge, any general principle which
would say whether the moduli space is ``supposed to be'' irreducible or not. On the one hand, the present case
is close to the abelian or K3 case, so it isn't too surprising if the moduli space remains irreducible; 
however on the other hand, at some point new irreducible components will be appearing as has been shown by
the first author in \cite{Mestrano}. So, we are led to analyse a number of cases for various aspects
of the argument. If any case is mistakenly ignored, it might hide a new irreducible component
which would then be missed. 

A natural question to wonder about is whether ``derived algebraic geometry'' could help here, but it would seem
that those techniques need to be further developed in order to apply to some basic geometric situations such as
we see here. Furthermore, each place in the argument where some case
is ruled out, constitutes a possible reason for there to be additional irreducible 
components in more complicated situations (such as on a sextic hypersurface). 
So, in addition to the theorem itself which only goes a little way
into the range that remains to be treated, much of the interest lies in the geometric situations
which are encountered along the way. 

\section{Notations}
\label{notations}

Throughout the paper, $X\subset \pp^3$ denotes a very general quintic hypersurface, and $E$ is a stable rank two vector
bundle of degree one\footnote{
This represents a change in notation from \cite{MestranoSimpson}, where
we considered bundles of degree $-1$. For the present considerations, bundles
of degree $1$ are more practical in terms of Hilbert polynomial. 
We apologize for this
inconvenience, but luckily the indexation by second Chern class stays
the same. Indeed, if $E$ has degree $1$ then $c_2(E)=c_2(E(-1))$ as can be seen
for example on the bundle $E=\Oo _X\oplus \Oo _X(1)$ with $c_2(E)=c_2(E(-1))=0$.
Thus, the moduli space of stable bundles $M_X(2,1,c_2)$ we look at here
is isomorphic to 
$M_X(2,-1,c_2)$ considered in \cite{MestranoSimpson}. 
}
with determinant $\bigwedge ^2E\cong \Oo _X(1)$ such that $c_2(E)=10$. 
The moduli space of stable bundles in general has been the subject of much work
\cite{Gieseker, GiesekerCons, GiesekerLi, HuybrechtsLehn,
Maruyama, MaruyamaET, MaruyamaTransform, Mukai, OGradyIrred,
OGradyBasic, YoshiokaAbelian, YoshiokaK3}, but the special case
$M_X(2,1,10)$ considered here goes into a somewhat new and uncharted territory.

Note that $Pic(X)=\zz$ is generated by $\Oo_X(1)$. The canonical
bundle is $K_X=\Oo _X(1)$. For any $n$ we  have $H^1(\Oo _X(n))=0$.
For $n\leq 4$ the map $H^0(\Oo _{\pp ^3}(n))\rightarrow H^0(\Oo _X(n))$ is
an isomorphism. 

The Hilbert polynomial of $E$ is $\chi (E(n))= 5n^2$. In particular $\chi (E)=0$.
We will be assuming that $E$ is general in some irreducible component of the moduli space.
From the previous paper \cite{MestranoSimpson} using
some techniques for bounding the singular locus which had also been introduced  
in \cite{Langer} and \cite{Zuo}, it follows that $E$ is unobstructed, so if $End^0(E)$ denotes the
trace-free part of $End(E)$ then $H^2(End^0(E))=0$. Note however that $H^2(\Oo _X)=\cc ^4$,
indeed it is dual to $H^0(\Oo _X(1))=H^0(\Oo _{\pp ^3}(1))$. Thus 
$$
H^2(E\otimes E^{\ast})\cong \cc ^4.
$$
The dual bundle is given by $E^{\ast}=E(-1)$, so duality says that $H^i(E(n))\cong H^{2-i}(E(-n))$.

The dimension of any irreducible component of the moduli space is the expected one, $20$. 
The subspace of bundles $E$ with $H^0(E)\neq 0$ has dimension $<20$, see our previous paper \cite{MestranoSimpson},
so a general $E$ has $H^0(E)=0$. It follows from duality that $H^2(E)=0$ and by $\chi (E)=0$
we get $H^1(E)=0$.  Throughout the paper (except at
one place in Section \ref{onquintic}), we consider only bundles with $H^0(E)=0$. 

Duality says that $H^2(E(1))$ is dual to $H^0(E(-1))=0$. 
Since $\chi (E(1))=5$, if we set $f:= h^1(E(1))$ then $h^0(E(1))=5+f$. In particular there are at least $5$ linearly
independent sections of $E(1)$ which may be viewed as maps $s:\Oo _X(-1)\rightarrow E$
or equivalently $s:\Oo _X\rightarrow E(1)$. 
Note that the zero set of $s$ has to be of codimension $2$, as any codimension-one component would be a divisor
integer multiple of the hyperplane class but $h^0(E)=0$ so this can't happen. 
If we  choose one such map $s$, then
we get the {\em standard exact sequence}
\begin{equation}
\label{standardexact}
0\rightarrow \Oo_X (-1)\rightarrow E \rightarrow J_P(2)\rightarrow 0
\end{equation}
and its twisted versions such as
\begin{equation}
\label{standardexact1}
0\rightarrow \Oo_X \rightarrow E(1) \rightarrow J_P(3)\rightarrow 0.
\end{equation}
As a notational matter, $J_P$ denotes the ideal of $P$ in $X$ or in $\pp^3$. Which one it is should
be clear from context, for example it is the ideal of $P\subset X$
in the above sequences, and we choose not to weigh down the notation with an extra subscript. 

Calculation of the Chern class $c_2(E)=10$ shows that $P\subset X$ is a subscheme of length $20$.
It is a union of possibly nonreduced points, which are locally complete intersections i.e.\  defined by two equations. 
Furthermore, as is classically well-known \cite{EisenbudGreenHarris} \cite{Iarrobino}, 
$P$ satisfies the Cayley-Bacharach condition for $\Oo _X(4)$
which we denote by $CB(4)$,
saying that any subscheme $P'\subset P$ of colength $1$ imposes the same number of conditions
as $P$ on sections of $\Oo _X(4)$.   The extension class is governed by an element of $H^1(J_P(4))^{\ast}$,
and we have the exact sequence
$$
0\rightarrow H^0(J_P(4))\rightarrow H^0(\Oo _X(4))\rightarrow \Oo _P(4)\rightarrow H^1(J_P(4))\rightarrow 0.
$$
To get a locally free $E$, the extension class should be nonzero on each vector coming from a
point in $P$, the existence of such being exactly $CB(4)$.
Note that $h^0(\Oo _X(4)) = 35$, and define $e:= h^1(J_P(4))-1$. Then $e\geq 0$ (the extension
can't be split, indeed this is part of the $CB(4)$ condition), and $h^0(J_P(4))=16+e$.

The ``well-determined'' case is when $e=0$. Then the extension class is well-defined up to a scalar multiple
which doesn't affect the isomorphism class of $E$, and the existence of the nonzero class in $H^1(J_P(4))$
is expected to impose $16$ conditions on the $20$ points,
giving the expected dimension of the Hilbert scheme of such subschemes $P\subset \pp ^3$
to be $44$. In Section \ref{seminatural}, we will show $f=0\Rightarrow e=0$ and in
that case the bundle $E$ has seminatural cohomology. 

Here are a few techniques which will often be useful throughout the paper. 

\begin{lemma}
If a zero-dimensional 
subscheme $P\subset X$ satisfies $CB(n)$ then it satisfies $CB(m)$ for any $m\leq n$. 
\end{lemma}
\begin{proof}
Suppose $P'\subset P$ has colength $1$. Choose a section $g\in H^0(\Oo _X(n-m))$ nonvanishing at
all points of $P$, then if $f\in H^0(J_{P'}(m))$ we have $fg\in H^0(J_{P'}(n))$.
By $CB(n)$, $fg$ vanishes on $P$, but $g$ is a unit near any point of $P$ so 
$f$ vanishes on $P$,  proving $CB(m)$.
\end{proof}

The results  of \cite{BGS} allow us to estimate the dimension of the Hilbert scheme of $0$-dimensional
subschemes of a curve, as was used in some detail in \cite{MestranoSimpson}. Mainly,
as soon as the curve is locally planar, the space of subschemes of length $\ell$ has
dimension $\leq \ell$. 

If $W\subset X$ is a divisor and $P$ is a zero-dimensional subscheme, we obtain the {\em residual subscheme}
$P^{\perp}$ of $P$ with respect to $W$, such that $\ell (P^{\perp}) +\ell (P\cap W)=\ell (P)$.
It is characterized by the property that sections of $\Oo _X(n)(-W)$ which vanish on $P^{\perp}$,
map to sections of $\Oo _X(n)$ vanishing on $P$. If $P$ is reduced then $P^{\perp}$ is
just the union of those point of $P$ which are not in $W$; if $P$ contains some
nonreduced schematic points then the structure of $P^{\perp}$ may be more complicated.

\begin{lemma}
\label{inquadric}
If $P$ satisfies $CB(3)$ and $P''\subset P$ is a subscheme of colength $2$,
suppose $P''$ is contained in a quadric. Then $P$ is contained in the same quadric.
\end{lemma}
\begin{proof}
The residual subscheme $P^{\perp}$ for the quadric has length $\leq 2$.
If it is nonempty, we can choose a linear form containing a subscheme of 
colength $1$ of $P^{\perp}$, corresponding to a subscheme 
$P^1\subset P$ of colength $1$. Applying $CB(3)$ to
the product of the quadric and the
linear form is a contradiction, so $P^{\perp}=\emptyset$ 
and we're done. 
\end{proof}

\section{Restriction to a plane section}
\label{planesection}

Suppose $H\subset \pp ^3$ is a hyperplane, and let $Y:=H\cap X$. 
By the genericity assumption on $X$ in particular $Pic(X)$ generated by $\Oo _X(1)$,
we get that $Y$ has to be reduced and irreducible. Its canonical sheaf is $\Oo _Y(2)$.
When $Y$ is smooth, then, it is a plane curve of degree $5$ and genus $6$. 
We have an exact sequence
$$
0\rightarrow E \rightarrow E(1)\rightarrow E_Y(1)\rightarrow 0.
$$
From the vanishing of $H^i(E)$ it follows that $H^2(E(1))=0$ (but this is also clear
from duality), and 
$$
H^0(E(1))\stackrel{\cong}{\longrightarrow} H^0(E_Y(1)),
$$
$$
H^1(E(1))\stackrel{\cong}{\longrightarrow} H^1(E_Y(1)).
$$

Suppose $L\subset \pp ^3$ is a line. A generic $X$ doesn't contain any lines, so $L\cap X$ is
a finite subscheme of length $\ell (L\cap X)=5$. We claim that for a general plane $H$ containing $L$,
the intersection $Y=H\cap X$ is smooth. This holds by Bertini's theorem away from the
base locus of the linear system of planes passing through $L$, so we just have to
see that it also holds at a point $x\in L\cap X$. Note that $T_xL\subset T_xX$ is a one-dimensional subspace.
A general $H$ will have tangent space which is a general plane in $T_x\pp ^3$
containing $T_xL$. Thus, a general plane $H$ containing $L$ has tangent space $T_xH$ which 
doesn't contain $T_xX$; in particular the intersection $T_xH\cap T_xX=T_xL$ is transverse.
This implies that $H\cap X$ is smooth at $x$. This works for all the finitely many points $x\in L\cap X$,
so the general section $Y=H\cap X$ is smooth. It is therefore a smooth plane curve of degree $5$ and
genus $6$. Notice that $L\subset H$ so $L\cap X \subset Y$. 

Pick $Y$ as in the previous paragraph, 
suppose $Q\subset L\cap X$ is a finite subscheme of length $4$, and suppose that $x\in H^0(E(1))$ is
a section vanishing on $Q$. Then $s|_Y$ is a section of $H^0(E(1))$ vanishing on $Q\subset Y$.
As $Y$ is smooth, the finite subscheme $Q$ is a Cartier divisor. The section $s|_Y$ corresponds
to a map $\Oo _Y\rightarrow E(1)$ which, since it vanishes on $Q$, gives a map
$$
\Oo _Y(Q)\rightarrow E(1).
$$
Let $Q'\subset Y$ be the divisor of zeros of $s$, in particular $Q\subset Q'$,
and $s$ extends to a strict map, i.e.\  an inclusion of a sub-vector bundle
$$
\Oo _Y(Q')\hookrightarrow E(1).
$$
The quotient line bundle is $\Oo _Y(3-Q')$
where the notation here combines $\Oo _Y(3)$ which is three times the hyperplane
divisor (which has degree $5$ on $Y$), with the divisor $Q'$. In particular $\Oo  _Y(3-Q')$ is a line
bundle of degree $15-\ell (Q')$. 
We obtain an exact sequence
$$
0\rightarrow \Oo _Y(Q')\rightarrow E_Y(1)\rightarrow \Oo _Y(3-Q')\rightarrow 0,
$$
leading to the long exact sequence of cohomology. This construction will be used many times
in Section \ref{baseloci}. 

Another useful construction is the following. Write $L\cap Y=x+y+u+v+w$, possibly
with some of the points being the same. Take a linear form containing $w$ as an isolated
zero, and divide by the equation of $L$. This gives a meromorphic function whose polar
divisor is $x+y+u+v$. Equivalently, $\Oo _Y(x+y+u+v)$ has a section nonvanishing
at the points $x,y,u,v$. This will be used often without too much further notice below.

\section{The seminatural condition}
\label{seminatural}

\begin{hypothesis}
\label{snhyp}
Assume that $h^0(E(1))=5$, and that $H^i(E)=0$ for $i=0,1,2$.
\end{hypothesis}

Recall that the second part
is true for any $E$ general in
its irreducible component as discussed above. 

The goal of this section is to show that \ref{snhyp}
implies $f=0$ and $E$ has seminatural cohomology,
which in this case means $H^0(E(n))=0$ for $n\leq 0$, $H^2(E(n))=0$ for $n\geq 0$,
and $H^1(E(n))=0$ for all $n$. Our main Theorem \ref{main} is the statement that there is only
a single irreducible
component corresponding to such bundles, so Hypothesis \ref{snhyp} will be in effect throughout
the rest of the paper.

\begin{lemma}
\label{sncase}
If $h^0(E(1))=5$ then $f=0$, in other words $H^1(E(1))=0$. If $s:\Oo (-1)\rightarrow E$
has scheme of zeros $P$, then saying $h^0(E(1))=5$ is equivalent to requiring that
$h^0(J_P(3))=4$, and saying that all $h^i(E)=0$ is equivalent to 
requiring that $h^0(J_P(2))=0$. 
\end{lemma}
\begin{proof}
As discussed above, $h^2(E(1))=0$ so the fact that $\chi (E(1))=5$ gives the first statement.
For the second statement, use the fact that $H^1(\Oo _X(n))=0$ for all $n$, and the
long exact sequences of cohomology for the extension $E(1)$ of $J_P(3)$ by $\Oo _X$
and similarly $E$ of $J_P(2)$ by $\Oo _X(-1)$. 
\end{proof}

A first part of the seminatural condition is easy to see.

\begin{lemma}
\label{seminat02}
Under our hypothesis \ref{snhyp}, $H^0(E(n))=0$ for $n\leq 0$, and $H^2(E(n))=0$ for $n\geq 0$.
\end{lemma}
\begin{proof}
Since $H^0(E)=0$ it follows that $H^0(E(n))=0$ for all $n\leq 0$,
and for $n\geq 0$, $H^2(E(n))$ is dual to $H^0(E(-n))=0$. 
\end{proof}

The main step towards the seminatural condition is the next twist:

\begin{proposition}
We also have $H^1(E(2))=0$.
\end{proposition}
\begin{proof}
If $Y=H\cap X$ is a smooth plane section, we claim $H^0(E_Y(-1))=0$.
If not, then we would get an inclusion $\Oo _Y(1)\hookrightarrow E_Y$,
hence $\Oo _Y(2)\hookrightarrow E_Y(1)$. However, $Y$ is a curve of genus $6$
and $K_Y=\Oo _Y(2)$ so $H^0(\Oo _Y(2))$ has dimension $6$. This gives
$h^0(E_Y(1))\geq 6$. Consider the exact sequence
$$
0\rightarrow E\rightarrow E(1)\rightarrow E_Y(1)\rightarrow 0.
$$
The fact that $H^1(E)=0$ implies that $H^0(E(1))$ surjects onto $H^0(E_Y(1))$,
so $h^0(E(1))\geq 6$. This is a contradiction to $h^0(E(1))=5$, showing that 
$H^0(E_Y(-1))=0$.

To show that $H^1(E(2))=0$, it suffices by duality to show that
$H^1(E(-2))=0$. Consider the exact sequence
$$
0\rightarrow E(-2)\rightarrow E(-1)\rightarrow E_Y(-1)\rightarrow 0.
$$
Again by duality from Lemma \ref{sncase}, $H^1(E(-1))=0$, 
so the long exact sequence gives an isomorphism between $H^0(E_Y(-1))$
and $H^1(E(-2))$. From the previous paragraph we obtain $H^1(E(-2))=0$. 
This proves the proposition.
\end{proof}

\begin{corollary}
\label{seminatcor}
Under Hypothesis \ref{snhyp}, $E$ has seminatural cohomology: $H^1(E(n))=0$ for all $n$.
\end{corollary}
\begin{proof}
By duality it suffices to consider $n\geq 0$ and we have already done $n=0,1,2$.
Consider the case $n=3$. This could be done by continuing as in the
previous proposition but here is another argument.
Choose an inclusion $s:\Oo (-1)\rightarrow E$,
and let $P$ be the subscheme of zeros of $s$. Choose a general plane section $Y=H\cap X$
such that $H$ passes through one point $z\in P$ in a general direction. 
Then $s|_Y$ has a zero at $z$, of multiplicity $m$ with $1\leq m\leq 5$.
Indeed, $P$ cannot contain a $6$-fold fat point whose length is $21$, because $P$ has length $20$.
Thus the multiplicity of a general plane section of $P$ at any point $z$ is $\leq 5$. 
The section $s$ restricted to $Y$ therefore induces a strict inclusion of vector bundles
from $\Oo _Y(m\cdot z)$ to $E(1)$,
hence an exact sequence, of the form
$$
0\rightarrow \Oo _Y(2+m\cdot z) \rightarrow E(3) \rightarrow \Oo _Y(5 - m\cdot z) \rightarrow 0.
$$
Note that the sub-line bundle has degree $10+m$ and the quotient line bundle has
degree $25-m$, so both of these have vanishing $H^1$ by duality. It follows that $H^1(E(3))=0$.
For any $n\geq 4$ a similar argument (but $Y$ doesn't even need to pass through a point of $P$)
shows that $H^1(E(n))=0$. 
\end{proof}

\begin{corollary}
It follows that $e=0$, which is to say that for any inclusion $s:\Oo (-1)\rightarrow E$,
if $P$ is the subscheme of zeros of $s$ then $h^0(J_P(4))=16$.
\end{corollary}
\begin{proof}
Choose an inclusion $s$ and consider the exact sequence
$$
0\rightarrow \Oo _X(1)\rightarrow E(2)\rightarrow J_P(4)\rightarrow 0.
$$
Notice that $H^2(\Oo (1)) = H^2(K_X)=\cc$. The long exact sequence of cohomology then reads
$$
0\rightarrow H^1(E(2))\rightarrow H^1(J_P(4))\rightarrow \cc \rightarrow 0,
$$
since $H^2(E(2))=0$ and $H^1(\Oo _X(n))=0$ for all $n$. The previous conclusion says the
term on the left $H^1(E(2))$ vanishes,
so $H^1(J_P(4))=\cc$. It is generated by the nonzero extension class governing
the exact sequence corresponding to $s$. On the other hand we have
$$
0\rightarrow J_P(4)\rightarrow \Oo _X(4)\rightarrow \Oo _P(4)\rightarrow 0
$$
so the map $H^0(\Oo _X(4)) = \cc ^{35}\rightarrow \Oo _P(4) = \cc ^{20}$ has cokernel
$H^1(J_P(4))$ of dimension $1$. It follows that the kernel $H^0(J_P(4))$ has dimension $16$. 
\end{proof}

\begin{corollary}
\label{snunique}
Pick a section $s\in H^0(E(1))$ and let $P$ be its subscheme of zeros. 
The extension class defining $E$ as an extension of $J_P(2)$ by $\Oo _X(-1)$ is unique up to a scalar.
\end{corollary}
\begin{proof}
Recall from where $e$ was defined that the space of extensions, $H^1(J_P(4))$, has dimension $e+1$.
Thus, the condition $e=0$ means that this is a line: the extension is unique up to scalars, and 
for a given subscheme $P$ there is a unique bundle extension $E$ up to isomorphism.
\end{proof}

\begin{corollary}
\label{H0EYzero}
If $Y=H\cap X$ is a plane section, then $H^0(E_Y)=0$. Also, $H^0(E(1))\stackrel{\cong}{\rightarrow} H^0(E_Y(1))$
and $H^1(E_Y(1))=0$. 
\end{corollary}
\begin{proof}
Consider the exact sequence 
$$
0\rightarrow E(-1) \rightarrow E\rightarrow E_Y \rightarrow 0.
$$
From $H^0(E)=0$ and $H^1(E(-1))=0$ we get $H^0(E_Y)=0$. Similarly, the exact sequence 
$$
0\rightarrow E \rightarrow E(1)\rightarrow E_Y(1) \rightarrow 0
$$
together with $H^i(E)=0$ gives $H^i(E(1))\stackrel{\cong}{\rightarrow} H^i(E_Y(1))$.
\end{proof}

\section{The structure of the base loci}
\label{baseloci}

Let $B_2\subset X$ be the subset of points where all sections of $H^0(E(1))$ vanish. 
Let $B_1\subset X$ be the subset of points $x$ such that the image of $H^0(E(1))\rightarrow E(1)_x$
has dimension $\leq 1$ (in particular $B_2\subset B_1$). These are the {\em base loci}
of sections of $E(1)$. In this section, we obtain some information about these base loci,
which will allow us to to deduce, in Section \ref{Pstructure}, 
that the zero-scheme of a general section $s$ has some
fairly strong general position properties. 

\subsection{There is at most one point in $B_2$}
\label{atmostone}

\begin{proposition}
\label{B2onepoint}
The subset $B_2$ has at most one point and if it exists, then the sections of $H^0(E(1))$
define this reduced point as a subscheme. 
\end{proposition}
\begin{proof}
Suppose $p\neq q$ are two points of $B_2$. Then all sections of $E(1)$ vanish at $p$ and $q$.
Consider a plane section $Y=H\cap X$ such that $p,q \in Y$, but $Y$ general for this property,
in particular $Y$ is smooth (Section \ref{planesection}). The map
$$
E(1)_p \oplus E(1)_q \rightarrow H^1(E_Y(1-p-q))
$$
is injective. Furthermore it is surjective since $H^1(E_Y(1))=0$. 

Let $L$ denote the line through $p$ and $q$. It intersects $Y$ in a divisor denoted $p+q+u+v+w$.
Some of the points $u,v,w$ may be equal or equal to $p$ or $q$. 

We have an exact sequence 
$$
0\rightarrow E_Y \rightarrow E_Y(1)\rightarrow E_{L\cap Y}(1)\rightarrow 0,
$$
and on the other hand, the exact sequence 
$$
0\rightarrow E(-1)\rightarrow E\rightarrow E_Y\rightarrow 0
$$
gives $H^1(E_Y)\stackrel{\cong}{\rightarrow} H^2(E(-1))\cong H^0(E(1))^{\ast}\cong \cc ^5$. 
Hence the image of $H^0(E_Y(1))\rightarrow E_{L\cap Y}(1)$ has codimension $5$, and
since $L\cap Y$ is a finite subscheme of length $5$, $E_{L\cap Y}(1)\cong \cc ^{10}$ 
so the image has dimension $5$ too. 

We may impose the condition of vanishing at two points $u,v$ and obtain a nonzero section $s\in H^0(E_Y(1)(-p-q-u-v))$. 
This has the required meaning when
some of the points coincide, using the previous paragraph.
However, the section $s$ then doesn't vanish
at the third point $w$, otherwise we would get a section in $H^0(E_Y(1)(-L\cap Y))=H^0(E_Y)$
contradicting Corollary \ref{H0EYzero}. 

This section generates a sub-line bundle $M\subset E_Y(1)$, with $M=\Oo _Y(p+q+u+v+D)$ for an effective divisor
$D$ not passing through $w$. Note that $\Oo _Y(p+q+u+v)$ has a nonzero section, corresponding to 
the quotient of a linear form (on the plane $H$) vanishing at $w$ but not along $L$, divided by a linear form
vanishing along $L$. If $D$ doesn't contain both $p$ and $q$ then we would get a section of $E(1)$ nonvanishing
at one of those points, contradicting our assumption $p,q\in B_2$. Therefore $D\geq p+q$. It follows that $w\neq p,q$.
The same reasoning works for $u$ and $v$ too, so $u,v,w$ are three points distinct from $p$ or $q$. 

Our section $s$ comes from a section in $H^0(E(1))$ corresponding to $\Oo (-1)\rightarrow E$, and the
subscheme of zeros $P$ contains $p,q,u,v$. These are four points on the line $L$, so any cubic form vanishing
at $P$ has to vanish along $L$. In particular, elements of $H^0(J_P(3))$ vanish at $w$. This implies that
elements of $H^0(E(1))$ evaluate at $w$ to elements in the line $M_w\subset E(1)_w$. Thus $w\in B_1$. 

This reasoning holds even if $w$ coincides say with $v$; it means that all elements of $J^0(J_P(3))$ have to
vanish in the tangent direction corresponding to the additional point $w$ glued onto $v$, which still gives
a rank one condition on the values of sections of $E(1)$ at the point $w$. 

The same reasoning holds for $u$ and $v$. 
If at least two of the points $u,v,w$ are distinct, then we obtain this way at least two points of $B_1$
along the line $L$. Then, vanishing at these two points consists of two conditions, so we can impose further
vanishing at the third point (even if it is a tangential point at one of the other two) and obtain a 
nonzero section which vanishes at all five points. As before this yields a nonzero section of $H^0(E_Y)$
contradicting Corollary \ref{H0EYzero}. 

It remains to consider the case when all three points are the same, that is to say $L\cap Y=p+q+3u$
with $u\neq p,q$, and 
choosing a section vanishing at $p,q$ and two times at $u$
generates a subbundle $M=\Oo _Y(ap+bq+2u +D)\hookrightarrow E(1)$ with $D$ an effective divisor distinct from $p,q,u$,
and $a,b\geq 2$. Recall that if either $a=1$ or $b=1$ then this would give a section of $E(1)$ nonvanishing at
$p$ or $q$ contradicting our assumption $p,q\in B_2$. 

As above, we have $u\in B_1$. Therefore, choosing a section in $H^0(E_Y(1)(-2u))$ represents only $3$
conditions rather than $4$, hence there are two linearly independent such sections $s_1,s_2$. We claim that
the values of these two sections, in $E_Y(1)(-2u)_u$, are  linearly independent. Indeed, otherwise
a combination of the two would vanish again at $u$ and this would give a section of
$E_Y(1)(-3u)$ which also vanishes at $p, q\in B_2$. This would give a nonzero element of $H^0(E_Y)$
which can't happen.

Let $a$ and $b$ be the smallest possible orders of vanishing of $s_i$ at $p$ and $q$ respectively,
and by linear combinations we can assume that both of them vanish to those orders. 
They give maps
$$
M_1=\Oo _Y(ap+bq+2u +D_1)\stackrel{s_1}{\rightarrow} E_Y(1), 
$$
$$
M_2=\Oo _Y(ap+bq+2u +D_2)\stackrel{s_2}{\rightarrow} E_Y(1),
$$
and the resulting map
$$
M_1\oplus M_2 \rightarrow E_Y(1)
$$
has image of rank $2$ at the point $u$ by the previous paragraph. Therefore it is injective. It follows
that ${\rm deg}(M_1\oplus M_2)\leq {\rm deg}(E_Y(1))= 15$. Suppose ${\rm deg}(D_1)$ is the smaller of the two,
then we get $a+b+2+{\rm deg}(D_1)\leq 7$. We may also by symmetry assume $a\geq b$. Write $M=M_1$ and $D=D_1$.
There are three possibilities:
$$
M = \Oo _Y (2p+2q+2u + d), \;\;\; D=(d),\;\;\; {\rm deg}(M)= 7
$$
$$
M = \Oo _Y (2p+2q+2u), \;\;\; D=0,\;\;\; {\rm deg}(M)= 6
$$
or
$$
M = \Oo _Y (3p+2q+2u), \;\;\; D=0,\;\;\; {\rm deg}(M)= 7.
$$

In each case, let $N:= E_Y(1)/M = \Oo _Y(3)\otimes M^{-1}$ be the quotient
bundle. Recall that $\Oo _Y(3)=\Oo _Y(3p+3q + 9u)$ and $K_Y=\Oo  _Y(2)=\Oo _Y(2p+2u+6v)$. 
We have an exact sequence 
$$
H^0(N)\rightarrow N_p\oplus N_q \rightarrow H^1(N(-p-q))\rightarrow H^1(N) .
$$
The rightmost map is dual to 
$$
H^0(K_Y \otimes N^{-1}) \rightarrow H^0(K_Y \otimes N^{-1} (p+q)).
$$
Notice however that $N^{-1}= \Oo _Y(-3)\otimes M$ so $K_Y\otimes N^{-1}= M(-1)=M(-p-q-3u)$. 
Hence our rightmost map is dual to
$$
H^0(M(-p-q-3u)) \rightarrow H^0(M(-3u)).
$$
This map is surjective; indeed, the condition
$p,q\in B_2$ means that all sections of $M$ must vanish at $p$ and $q$, 
and sections of $M(-3u)$ are in particular sections of $M$, so every 
element of $H^0(M(-3u))$ must come from an element of $H^0(M(-p-q-3u))$.
This surjectivity translates by duality to the statement that the
rightmost map in the above exact sequence, is injective. It follows that
$H^0(N)\rightarrow N_p\oplus N_q$ is surjective. 

In other words, the values of global sections of $N$ at $p$ and $q$ span a
two-dimensional space. Since on the other hand the values of 
sections of $E_Y(1)$ must vanish at $p$ and $q$, this implies from
the exact sequence
$$
H^0(E_Y(1))\rightarrow H^0(N) \rightarrow H^1(M)
$$
that we have $h^1(M)\geq 2$. 

Consider now the three cases, the first case being 
$M = \Oo _Y (2p+2q+2u + d)$, with 
$\chi (M)= 2$, so $h^1(M)\geq 2$ implies that $h^0(M)\geq 4$. 
Vanishing at $2u$ imposes two conditions which leaves 
$h^0(M(-2u))= h^0(\Oo _Y(2p+2q+d) \geq 2$. These sections
must vanish at $p$ and $q$, so we get  
$h^0(\Oo _Y(p+q+d) \geq 2$. Now, our two independent sections of $\Oo _Y(p+q+d)$
cannot vanish at both $p$ and $q$
because $Y$ is not $\pp^1$ so there are no functions with a single nontrivial pole at $d$. 
We get a section of $\Oo _Y(p+q+d)$ whose value at one of $p$ or $q$ is nonzero. 
Multiplying this by the section of $\Oo _Y(p+q+2u)$ nonvanishing at $p$ and $q$, gives a section
of $M$ nonvanishing at $p$ or $q$, a contradiction which treats the first case. 

In the next case, $M=\Oo _Y(2p+2q+2u)$ with $\chi (M) = 1$ so $h^1(M)\geq 2$ gives
$h^0(M)\geq 3$. As usual, sections of $M$ have to vanish at $p$ and $q$ so 
$h^0(M(-p-q))=h^0(\Oo _Y(p+q+2u))\geq 3$. But notice that $\Oo _Y(p+q+2u)=\Oo _Y(1)(-u)$.
The map $\cc ^3=H^0(\Oo _H(1))\rightarrow H^0(\Oo _Y(1))$ is an isomorphism; and
$\Oo _Y(1)$ is generated by its global sections. Hence, vanishing of a section at $u$
imposes a nontrivial condition, giving $h^0(\Oo _Y(1)(-u))=2$. This contradicts the
previous estimation of $\geq 3$. This contradiction completes this case.

In the last case, $M=\Oo _Y(3p+2q+2u)$ with $\chi (M) = 2$ so $h^1(M)\geq 2$ gives
$h^0(M)\geq 4$. This is similar to the first case. 
Vanishing at $2u$ imposes two conditions, and then the sections must further vanish at $p$ and $q$, 
which leaves 
$h^0(M(-2u))= h^0(M(-p-q-2u))= h^0(\Oo _Y(2p+q+d) \geq 2$. 
If we have a section here which is nonzero at either $p$ or $q$, then 
multiplying it by the section of $\Oo _Y(p+q+2u)$ nonvanishing at $p$, gives a section
of $M$ nonvanishing at $p$ or $q$, a contradiction. Therefore,
both sections in  $H^0(\Oo _Y(2p+q+d)$ have to vanish further at $p$ and $q$. This would give
$h^0(\Oo _Y(p+d))\geq 2$. That can happen only if $Y$ is a hyperelliptic curve.

But a smooth plane curve of degree $5$ is never 
hyperelliptic. If $Y$
were hyperelliptic, for a general $y\in Y$ let $y'$ be the conjugate by the involution;
there is a meromorphic function $f$ with polar divisor $y+y'$. 
The line $L$ through $y$ and $y'$ meets $Y$ in $5$ different points,
otherwise the map sending $L$ to the point of higher multiplicity would be a $\pp ^1\rightarrow Y$.
Write $L\cap Y = y+y'+u+v+w$. A linear
form vanishing at one of the other points, say $w$, divided by the equation  
of $L$, gives a meromorphic function $g$ whose polar locus is $y+y'+u+v$. 
Then $[1:f:g]$ provides a degree $4$ map to $\pp ^2$. By looking at the genus,
it can't be injective, also it spans the plane so it isn't a degree $4$ map to a line.
The only other case would be a degree $2$ map to a conic. But in that case, 
a  linear combination of $f$ and $g$ would have polar divisor $u+v$. 
Doing the same for the other possibilities gives a $3$ dimensional space of functions
with poles $\leq u+v+w$, but $Y$ can't have a degree $3$ spanning map to $\pp ^2$. 
This shows that $Y$ cannot be hyperelliptic, so this case is also ruled out. 

We have now finished showing that it is impossible to have two distinct points $p,q\in B_2$.
The same proof works equally well if $q$ is infinitesimally near $p$; this double point defines
a tangent direction, and $L$ should be chosen as the tangent line in this direction. 
The main case as before is when $L\cap Y = 2p+3u$, and as before we get three cases:
either $M=\Oo _Y(4p+2u+d)$,  $M= \Oo _Y(4p + 2u)$, or $M= \Oo _Y(5p+2u)$. 
The main principle here is that sections of $M$ have to vanish on both $p$ and the nearby point $q$,
that is to say they have to vanish to order $2$ at $p$. With this, the same proofs as above hold,
so this shows that if $B_2$ is nonempty, then it is a single reduced point. 
This completes the proof of the proposition. 
\end{proof}

\subsection{Local structure of $B_1$ at a point of $B_2$}
\label{localstrucB2}

For the next discussion, we assume that there is a point $p'\in B_2$, unique by above.
Consider the schematic structure of $B_1$ around this point $p'$. An argument similar
to the above, allows us to show that $B_1$ can't contain the third infinitesimal
neighborhood of $p'$; however, we haven't been able to rule out the possibility
that it might contain the second neighborhood. We will formulate this statement
precisely in the form of the following
lemma, even though we haven't really defined the schematic structure of 
$B_1$. Recall that $H^0(E(1))\stackrel{\cong}{\rightarrow} H^0(E_Y(1))$. 

\begin{lemma}
Suppose $Y\subset X$ is a general plane section passing through $p'$. Choose
$s\in H^0(E_Y(1))$, vanishing to order $1$ at $p'$. Let $M\subset E_Y(1))$ be
the sub-line bundle generated by $s$. Then there exists a section $t\in H^0(E_Y(1))$
such that the projection of $t$ as a section of $N:= E_Y(1)/M$ vanishes to order
at most $2$ at $p'$. 
\end{lemma}
\begin{proof}
Suppose on the contrary that all sections vanish to order $\geq 3$ in $N$. 
As this is true on a general $Y$, we may also specialize $Y$ and it
remains true. In particular, choose a tangent line $L$ to $X$ at $p'$
such that the second fundamental form vanishes. Choose a smooth plane
section $Y$ corresponding to a plane containing $L$. Then $Y\cap L$ 
is a divisor of class $\Oo _Y(1)$, and we can write
$$
Y\cap L = 3p' + u+v,
$$
where, as far as we know for now,
$u$ and $v$ might be the same, and one or both might be equal to $p'$.

Choose a nonzero section $s\in H^0(E_Y(1))$ which has only a simple zero at $p'$.
Recall that this is possible, by the result that $B_2$ is reduced in the previous
proposition. Let $M\subset E_Y(1)$ be the subline bundle generated by $s$,
and let $N:=E_Y(1)/M$ be the quotient. The contrary hypothesis says that
all sections of $E_Y(1)$ vanish to order $3$ at $p'$, when projected into $N$.
This means that the condition of a section vanishing to order $3$ at $p'$
imposes only two additional conditions. Indeed, etale-locally we can
choose a basis for $E_Y(1)$ compatible with the subbundle $M$,
and impose two conditions stating that the first coordinate (corresponding
to $M$) vanishes to order $3$
(it automatically vanishes to order $1$ already). This implies that
the section vanishes, since the second coordinate vanishes to order $3$ by
hypothesis. 

Now since $h^0(E_Y(1))=5$, we can impose two further conditions and obtain
a section $t$ vanishing at $u$. The divisor of  vanishing of $t$ is therefore
$ap'+u+D$ where $a\geq 3$. If $a=3$ then we would get a morphism
$$
\Oo _Y(3p'+u)\rightarrow E_Y(1)
$$
nonzero at $p'$, but the line bundle $\Oo _Y(3p'+u)$ has a section nonvanishing
at $p'$, and this would contradict $p'\in B_2$. Therefore we can conclude
that $a\geq 4$. 

We now note that $u$ and $v$ must be distinct from $p'$. For example, if
$Y\cap L=4p'+u$, choose a section $s$ vanishing at $u$, and as described
above, we can assume vanishing to order $3$ at $p'$ which imposes two additional
conditions. If $s$ vanishes to order $\geq 4$ at $p'$ this
would give a section in $H^0(E_Y)$ which can't happen,
so we can assume that $M=\Oo _Y(3p'+u)$ and again this has a section
nonvanishing at $p'$, contradicting $p'\in B_2$. So, this case can't happen. 

Similarly if $Y\cap L=5p'$, vanishing to order $3$ imposes two conditions,
and vanishing to order $4$ imposes two more conditions so again there is a 
section $s$ which generates $M=\Oo _Y(4p')\subset E_Y(1)$, but this $M$
has a section nonvanishing at $p'$ contradicting $p'\in B_2$.
From these arguments we conclude that $u,v$ are different from $p'$. 

Next, use the fact that a cubic polynomial on $L$ vanishing at $4$
points in $L\cap X$, must also vanish on the fifth point. Suppose first
that $u\neq v$. Our section $t$ viewed as a section of $E(1)$ defines
a zero-scheme $P$, which contains its zeros on $Y$. In particular,
$P$ contains the scheme $3p'$ on $L$ as well as the scheme $u$.
Note on the other hand that $v\not \in P$ otherwise we would get a
section in $H^0(E_Y)$. We conclude that any element of $H^0(J_P(3))$
has to vanish at $v$. It follows that $v\in B_1$. By symmetry,
we get also $u\in B_1$. Now, vanishing of sections at $u$ and $v$
imposes $2$ conditions, and vanishing at $3p'\subset Y$ imposes
$2$ conditions as discussed above. This gives a section 
in $H^0(E_Y(1))$ vanishing at all of $Y\cap L$, hence a nonzero
section of $H^0(E_Y)$. We get a contradiction in this case.

To finish the proof of the lemma, we have to treat the case where $u=v$,
that is $Y\cap L = 3p'+2u$. As described previously, $u\neq p'$. 
Basically the same argument as before gives $u\in B_1$. 
Indeed, we can consider a section $t$ which vanishes at $3p'+u$.
It can't have a zero of order $2$ at $u$ otherwise we would get $H^0(E_Y)\neq 0$.
Let $M\subset E_Y(1)$ be the subline bundle generated by $t$,
and let $N$ be the quotient. Write $M=\Oo _Y(ap'+u+D)$
with $a\geq 4$ and $D$ disjoint from $p',u$.
We may also consider $t$ as a section
defined over $X$, inducing a quotient morphism $E(1)\rightarrow \Oo _X(3)$.
When restricted to $Y$ this provides a morphism $E_Y(1)\rightarrow \Oo _Y(3)$
which is the same as the map to $N$ generically. Hence it must factor
through $E_Y(1)\rightarrow N \rightarrow \Oo _Y(3)$.
This is more precisely given by $N=\Oo _Y(3)(-ap'-u-D)$. 
Sections of $E(1)$ map to sections of $\Oo _X(3)$ vanishing on
the zero locus $P$ of $t$, which contains $3p'+u\subset L$ (a subscheme of length $4$). 
These sections must vanish on all of $L$, hence they vanish on $2u$. 
Thus the image of any section in $N$ has to be a section of 
$\Oo _Y(3)(-ap'-2u-D)=N(-u)$. This means that the sections of $E(1)$, evaluated
at $u$, must lie in $M_u$. In other words, $u\in B_1$ as claimed.

Vanishing
of a section at $u$ therefore imposes a single condition. So there
are two linearly independent sections $t_1,t_2$ which vanish
at $3p'+u$. No nonzero linear combination of these can have
a zero of order $2$ at $u$.
It follows that the derivatives of $t_1$ and $t_2$ at $u$ are linearly
independent. Let $M_1$ and $M_2$ denote the sub-line bundles of $E_Y(1)$
generated by the $t_i$. We have 
$$
M_i =  \Oo _Y(a_ip'+u + D_i)
$$
with $a_i\geq 4$ and $D_i\cap u=\emptyset$. But the line bundle $\Oo _Y(3p'+u)$
has a section nonvanishing at $u$, so $M_i$ has a section nonvanishing at $u$.
But the $M_i(u)\subset E(1)_u$ are generated by the derivatives of $t_i$,
which are linearly independent. Thus the $M_i(u)$ generate $E(1)_u$.
But as there are sections of $M_i$ nonvanishing at $u$, this contradicts $u\in B_1$. 
This completes the proof of the lemma. 
\end{proof}

\begin{corollary}
\label{lengthtwoatB2}
Suppose $p'\in B_2$. Then for a general section $s\in H^0(E(1))$,
the scheme of zeros of $s$ locally at $p'$ 
is either the reduced point $p'$,
or a length $2$ subscheme (infinitesimal tangent vector) at $p'$.
\end{corollary}
\begin{proof}
From the proposition before, the sections of $E(1)$ define
$p'$ as a reduced subscheme. This means that for any tangent
direction, there is at least one section whose derivative in that
direction doesn't vanish. So, if $Y\subset X$ is a generic curve through
$p'$, then the zero scheme $P$ of a general section $s$
has $P\cap Y=\{ p'\}$ being a reduced subscheme locally at $p'$.
It follows that $P$ is curvilinear at $p'$.

Assume that the general $P$ has length $\geq 3$ locally at $p'$.
Consider the two sections $s,t$ given by the previous lemma. Their zero
sets are therefore curvilinear subschemes of length $\geq 3$ at $p'$.
Given $s$, we may choose $t$ general, then the zero set of $t$
is transverse to that of $s$ at $p'$. For otherwise this would
mean that the tangent directions of the zero sets are always the
same, but that would give an infinitesimal tangent vector in $B_2$
contradicting the above proposition. So these curvilinear subschemes
are transversal. We may choose local coordinates at $p'$
so that they go along the coordinate axes, up to order $3$ at least. 
If $x,y$ are these coordinates with $p'=(0,0)$, we may write 
$$
s=xa, \;\;\; t=yb \mbox{  modulo terms of order }3
$$
where $a$ and $b$ are sections of $E(1)$ nonvanishing at $p'$.
Furthermore, $Y$ is transverse to $(x=0)$. We may assume that the sub-line bundle
of $E_Y(1)$ generated by $s$ is generated by $a|_Y$.

Notice that if $b(0,0)$ is linearly independent from $a(0,0)$ then
$s+t=xa+yb$ would be a section whose zero scheme is the reduced point 
$p'$ so we would be done. Therefore we may assume, after possibly
multiplying by a scalar, that $b(0,0)=a(0,0)$. 

The conclusion of the lemma
says that $t$ is not a section of this subline bundle to order $3$,
which means that $b|_Y$ is not a  multiple of $a$ to order $2$
i.e.\  modulo quadratic terms. We may therefore write
$$
b= a + xb_x + yb_y + \ldots
$$
with one of $b_x,b_y$ nonzero modulo $a(0,0)$. Look at the 
section $s+\lambda t$ for variable $\lambda \in \cc$;
its leading term is $(x+\lambda y)a(0,0)$.  
By our contrary hypothesis, we suppose
that the zero set of this section is curvilinear to order $3$ for all $\lambda$, which
means that there is a factorization of $s+\lambda t$
as a multiple of a single section of $E(1)$, up to terms of order
$3$. The first term has to be $(x+\lambda y)a(0,0)$, so we can write
$$
s+\lambda t =  (x+\lambda y + q(x,y))(a + xf_x + yf_y).
$$
This expands to 
$$
xa + \lambda y (a + x b_x + yb_y) = (x+\lambda y + q(x,y))(a + xf_x + yf_y)
$$
or simplifying (always modulo terms of order $3$),
$$
\lambda y (x b_x + yb_y) = q(x,y)a + (x+\lambda y)(xf_x + yf_y).
$$
Now compare terms modulo the section generated by $a$; we get
$$
f_x = 0, \,\, f_y = b_y \mbox{ modulo } a(0,0)
$$
and from the $xy$ term we get $\lambda b_x = f_y$ again modulo $a(0,0)$.
Putting these together gives that $\lambda b_x=b_y$ modulo $a(0,0)$, for all 
$\lambda$. This is possible only if $b_x$ and $b_y$ are multiples of $a(0,0)$;
but the conclusion of the previous lemma said that this wasn't the case. 
This contradiction completes the proof of the corollary. 
\end{proof}

The above discussion may seem somewhat complicated: let us explain the
geometric picture, in terms of a schematic notion of the base locus $B_1$.
The problem is that $B_1$ could have some ``layers'' surrounding the point $p'\in B_2$.
Locally, this would mean that the subsheaf of $E(1)$ generated by global sections, looks
like a rank $1$ subsheaf over $B_1$, the layers of which would give a certain
infinitesimal neighborhood of $p'$. In the lemma, we say that if we cut by a general
plane section $Y$ going through $p'$, then the intersection with $B_1$ has 
length at most $2$. Intuitively this means that while $B_1$ might have a single layer
around $p'$, it can't have two layers. Notice that in some directions $B_1$ might
be bigger, but in a general direction it has length $2$. Then, in the corollary,
we say that if the general section has a curvilinear zero set of length $\geq 3$,
that would mean that $B_1$ had to have at least two layers around $p'$.

\subsection{Dimension of the CB-Hilbert scheme}
\label{sec-dimensions}

Let ${\bf H}_X$ denote the Hilbert scheme of subschemes $P\subset X$ which satisfy $CB(4)$.
Let ${\bf H}_{\pp ^3}$ denote the Hilbert scheme of subschemes $P\subset \pp ^3$ which satisfy $CB(4)$.
We call these the {\em CB-Hilbert schemes}. 

Let ${\bf H}_X^{sn}$ and ${\bf H}^{sn}_{\pp ^3}$ denote the subschemes parametrizing $P$ such that
$h^0(J_P(3))= 4$ and $h^0(J_P(2))=0$, and (in the second case)
such that $P$ is contained in at least one smooth quintic surface. 
In that case, as we have seen in Lemma \ref{sncase}, any bundle $E$ extending $J_P(2)$ by
$\Oo _X(-1)$ has seminatural cohomology, so we call them the {\em seminatural CB-Hilbert schemes}. 
Furthermore, as in Corollary \ref{snunique}, the isomorphism class of $E$ is uniquely determined by $P$. 
Since $E$ is stable, it doesn't have any nontrivial automorphisms. 

\begin{proposition}
\label{dimensions}
The seminatural CB-Hilbert scheme ${\bf H}_X^{sn}$ has pure dimension $24$; the 
seminatural CB-Hilbert scheme ${\bf H}^{sn}_{\pp ^3}$ has pure dimension $44$. 
Denote by ${\bf H}_X^{sn}[2]$ and ${\bf H}^{sn}_{\pp ^3}[2]$ the fiber bundles over these,
parametrizing pairs $(P,U)$ where $P$ is a seminatural CB Hilbert point, and $U\subset H^0(J_P(3))$
is a $2$-dimensional subspace. These have pure dimensions $28$ and $48$ respectively.
\end{proposition}
\begin{proof}
A point in ${\bf H}_X^{sn}$ corresponds to a choice of bundle $E$ in $M^{sn}_X(2,1,10)$
plus a section $s\in H^0(E(1))$ up to scalar. As the moduli space has dimension $20$ 
and, for the seminatural case, $\pp H^0(E(1))$ has dimension $4$, the total dimension
of ${\bf H}_X^{sn}$ is $24$. The Hilbert scheme of pairs $(P,X)$ with 
$P\in {\bf H}_X^{sn}$ fibers over the $55$ dimensional space of quintics $X$ 
(note that $h^0(\Oo _{\pp ^3}(5))=56$) with $24$-dimensional fibers, so it
has dimension $79$. On the other hand, for a fixed $P\in {\bf H}^{sn}_{\pp ^3}$,
the space of quintics $X$ containing $P$ is $\pp H^0(J_P(5))$.
Notice that if $P$ is contained in at least one $X$ then the discussion of Section \ref{seminatural}
implies that $h^1(J_P(5))=0$ so $h^0(J_P(5))= 36$ and the space of quintics containing $P$
is an open subset of $\pp ^{35}$. So, the dimension of the Hilbert scheme ${\bf H}^{sn}_{\pp ^3}$
is $79-35=44$. The fiber bundles parametrizing choices of $U\subset H^0(J_P(3))$
are bundles of Grassmanians of dimension $4$, so they have dimensions $28$ and $48$ respectively. 
All irreducible components have the same dimension because the same discussion works for all of them. 
\end{proof}

\subsection{The base locus $B_1$ has dimension zero}

\begin{proposition}
\label{B1dimzero}
Suppose $E$ is a general point of its irreducible component. 
The subset $B_1$ of points at which $E(1)$ is not generated by global sections, has dimension $0$.
Equivalently, if $s$ is a general section of $E(1)$ and $P$ its subscheme of zeros, then
the base locus in $X$ of the linear system of cubics $H^0(J_P(3))$, has dimension $0$
(it remains possible that the base locus in $\pp^3$ could have dimension $1$, indeed
that will be a major case treated in Section \ref{sec-ccc} below). 
\end{proposition}
\begin{proof}
The proof takes up the rest of this section, using three further lemmas. 
Note first the equivalence of the two formulations. The section $s$ generates a rank one subsheaf
of $E(1)$ at all points outside $P$. Thus, if $B_1$ had positive dimension, this would mean that
all sections restrict to multiples of $s$ on $B_1$, so all sections of $H^0(J_P(3))$
would factor as a function vanishing on $B_1$ times some other function. So the second
statement implies the first. In the other direction, suppose all elements of the linear
system factored as $fg$ where $g$ is a fixed form, either linear or quadratic. Then
the zero set of $g$ would provide a positive dimensional component of $B_1$. 

To be proven, is that the elements of the linear system $H^0(J_P(3))$ cannot all
share a
common factor $g$. Suppose to the contrary that they did, and let $W\subset X$ be the
zero-set of $g$. It is a divisor either in the linear system $\Oo _X(1)$ or $\Oo _X(2)$,
which is to say that it is either a plane section or a conic section of $X$. 

Let $P^{\perp}$ be the residual subscheme of $P$ along $W$ (i.e.\  roughly speaking $P-P\cap W$). 
Recall that we are assuming
that $P$ is not contained in a conic section, so $P\not\subset W$
and $P^{\perp}$ is nonempty. 
The statement that elements of $H^0(J_P(3))$ vanish along $W$, means
that the map
$$
H^0(J_{P^{\perp}}(3)(-W))\rightarrow H^0(J_P(3))
$$
is an isomorphism. Recall also that the right hand side has dimension $4$
in our situation, so we get $h^0(J_{P^{\perp}}(3)(-W))=4$ too. 

It is now easy to rule out the case where $W$ is a conic section. 
Indeed, in that case we would have $h^0(J_{P^{\perp}}(1))=4$,
but $h^0(\Oo _X(1))=4$ and the space of sections generates $\Oo _X(1)$ everywhere,
so there are at most $3$ sections vanishing on a nonempty subscheme $P^{\perp}$
giving a contradiction.

Therefore, we may now say that $W$ is a plane section of $X$. 
From above, $h^0(J_{P^{\perp}}(2))=4$. 

Next, we claim that $P^{\perp}$ satisfies $CB(3)$, that is Cayley-Bacharach for $\Oo _X(3)$
which is the same as $\Oo _X(4)(-W)$.
Indeed, if $f$ is a section of $\Oo _X(3)$ and $g$ is the equation of $W$ then
$fg$ is a section of $\Oo _X(4)$. Suppose $P^3\subset P^{\perp}$ is a colength $1$ subscheme.
Then it induces a colength $1$ subscheme $P'\subset P$ such that $P^3$ is the residual of $P'$,
notice that $\Oo _{P^{\perp}}$ may be viewed as the ideal $(g)$ inside $\Oo _P$ so an ideal
of length $1$ in $\Oo _{P^{\perp}}$ gives an ideal of length $1$ in $\Oo _P$. 
Now if $f$ vanishes on $P^3$ then $fg$ vanishes on $P'$, so by $CB(4)$ for $P$
we get that $fg$ vanishes on $P$ which in turn says that $f$ vanishes on $P^{\perp}$. 
This proves that $P^{\perp}$ satisfies $CB(3)$ as claimed. It follows that $P^{\perp}$ also
satisfies $CB(2)$. 

The next remark is that $P^{\perp}$ is not contained in a plane, for if it were then the
union of this plane with the one defining $W$ would be a conic containing $P$,
contrary to our situation. 

We have the following lemma, which is a preliminary
version of the structural result of Proposition \ref{structureofP} below. Notice that here we haven't yet
shown that $B_1$  has dimension $0$, so we use the specific current situation in
the proof instead.  

\begin{lemma}
\label{Pdecomp}
In the situation of the proof of the present proposition, 
consider a general section $t\in H^0(E(1))$, and let $P\subset X$ be its subscheme of
zeros. Then $P$ decomposes as a disjoint union $P=P'\sqcup P''$ such that 
$P''$ is reduced, and $P'$ is either empty, consists of a point $p'$,
or an infinitesimal tangent vector at $p'$, in the latter two cases $p'$
is the unique point of $B_2$.
\end{lemma}
\begin{proof}
Choose first any section $s\in H^0(E(1))$ with zero-scheme $P$,
corresponding to
a subsheaf $\Oo _X(-1)\subset E(1)$. Let $r$ be another section linearly
independent from $s$, and let $F\subset E(1)$
be the subsheaf generated by $r$ and $s$. Let $K:= E(1)/F$ be the quotient. 
Let $\tilde{r}$ be the image of $r$ considered as a section of $J_P(3)$. 
Under our hypothesis of the proof of the proposition, the zero scheme $Z(\tilde{r})$
decomposes as $W\cup D$ where $D$ is a conic section. For $r$ sufficiently general,
$D$ doesn't contain $W$ (otherwise the conic section $2W$ would be a common zero
of the linear system, and we have ruled that out). Thus, $Z(\tilde{r})$
is smooth on the complement of a finite set. Now, we can choose
$t$ so that it is nonvanishing at all isolated points of $B_1$ (except maybe $p'\in B_2$),
and at the finite set of singularities. Thus, if the zero scheme $P(t)$ of $t$
meets $W$, it meets it at a point where $Z(\tilde{r})$ is reduced (which we think
of heuristically, as points where $B_1$ is reduced even though we haven't
given a scheme structure to $B_1$). Furthermore, since $P(t)$ moves (except
maybe at $p'\in B_2$), its intersection with $W$ is reduced. On the other hand,
at points located on $W$, $P(t)$ has to be locally contained in $W$, otherwise
we could add a small multiple of $r$ to split off the part of the subscheme sticking out of $W$. 
These together imply that the points of $P(t)$ contained in $W$ are reduced (except possibly
at $B_2$). The points outside of $W$ are reduced because they are located at places
where $E(1)$ is generated by global sections, again with the possible exception of $p'\in B_2$.
From the discussion of the above subsection, the local structure of $P(t)$ near 
the possible single point $p'\in B_2$ is at most an infinitesimal tangent vector of length $2$.
This gives the claim of the lemma.
\end{proof}

\begin{lemma}
The length of $P^{\perp}$ is $\leq 10$. 
\end{lemma}
\begin{proof}
Consider first the case where $B_2$ is empty, or $P$ is reduced at $p'\in B_2$,
or else $p'\not \in W$. In this case, $P=P^W\cup P^{\perp}$ with $P^W=P\cap W$ a
reduced subscheme. For fixed $W$, the dimension of the space of choices of $P^W$ is $\leq \ell (P^W)$.
On the other hand, $P^{\perp}$ is located at the intersection $C_1\cap C_2\cap X$ where $C_1$ and $C_2$
are conics, whose intersection has dimension $1$. Furthermore, no component of $C_1\cap C_2$ 
is contained in $X$, indeed the former has degree $4$ while curves in $X$ have degre $\geq 5$
because of the condition $Pic(X)=\langle \Oo _X (1)\rangle$. Therefore, $C_1\cap C_2\cap X$
is a finite set. Since $P^{\perp}$ is reduced except for a possible tangent vector at the unique point
$p'\in B_2$, we get that the dimension of the set of choices of $P^{\perp}$ for a given $C_1, C_2$,
is $\leq 1$. On the other hand, suppose $C_1,C_2,C_3$ are three general conics through $P^{\perp}$.
If their intersection is finite, it contains at most $8$ points; but with $\ell (P^{\perp})\geq 11$
this can't happen and
we must have a nontrivial curve in the intersection; this means that a double intersection
$C_1\cap C_2$ has to split into two pieces. The dimension of the space of such double
intersections is the dimension of the Grassmanian of $2$-planes in $H^0(\Oo (2))=\cc ^{10}$,
this Grassmanian has dimension $16$. However, as may be seen by a calculation of the
possible cases of splitting, the subvariety of the Grassmanian corresponding to double intersections
which split into at least two components, is $\leq 14$. Together with the possible one dimensional
choice of tangent vector at $p'$, we get altogether that the space of choices for $P^{\perp}$ together with
the two-dimensional subspace spanned by $C_1,C_2$, is $\leq 15$. Putting in $P^W$, we get that
the dimension of the space of choices of $P$ plus a $2$-dimensional subspace of $H^0(J_P(3))$,
is less than $15+3+\ell (P^W)$. The $3$ is for the space of choices of plane section $W$.
Now if $\ell (P^{\perp})\geq 11$ then $\ell (P^W)\leq 9$ and this dimension is $\leq 27$.
The dimension of the corresponding bundle over the seminatural CB-Hilbert scheme ${\bf H}_X^{sn}[2]$ is
$28$ (Proposition \ref{dimensions}) so such a bundle $E$ cannot be general in its irreducible
component. 

We are left to treat the case where the unique point $p'\in B_2$ lies on $W$, and
$P$ includes a tangent vector here. Let $P^1$ denote
the subscheme of $P$ located set-theoretically along $W$, and $P^2$ the complement.
Then the dimension of the space of choices of $P^1$ is still $\ell (P^1)$,
and the same argument as above gives that the dimension of the space of choices of 
$P^2$ plus a two-dimensional subspace of conics, is $\leq 15$. We get as before $\ell (P^2)\leq 10$.
On the other hand, the tangent vector at $p'$ might go outside of $W$ and contribute to $P^{\perp}$.
If the tangent vector stays inside $W$ then $P^{\perp}=P^2$ and we are done. If the tangent
vector goes outside of $W$, then the estimate from above 
says only that $\ell (P^{\perp})\leq 11$; however, we get an additional condition saying that 
the conics have to vanish at this point $p'\in W$, and this condition (which may be
seen, for example, as a condition on the choice of $P^1$ once $P^2$ and the conics are fixed)
gets us back to
the estimate $\ell (P^{\perp})\leq 10$. 
\end{proof}

\begin{lemma}
There is a plane section $V\subset X$ such that $V\cap P^{\perp}$ has length $\geq 5$.
\end{lemma}
\begin{proof}
Suppose not, that is to say, suppose that any plane section meets $P^{\perp}$
in a subscheme of length $\leq 4$. In order to obtain a contradiction,
we show that under this hypothesis, $\ell (P^{\perp})\geq 11$. 

Choose a plane section meeting $P^{\perp}$ in a subscheme of length $\geq 3$,
call this intersection $P^{\perp}_+$ and let $P^{\perp}_-$ denote the residual subscheme. 
Then the condition $CB(3)$ for $P^{\perp}$ implies $CB(2)$ for $P^{\perp}_-$. 
The results of our previous paper \cite{MestranoSimpson} therefore apply:
\newline
(a) $\ell (P^{\perp}_-)\geq 4$;
\newline
(b) if $\ell (P^{\perp} _-)=4$ or $5$ then $P^{\perp}_-$ is contained in a line;
\newline
(c) if $\ell (P^{\perp}_-)=6$ or $7$ then $P^{\perp}_-$ is contained in a plane.

However, our hypothesis for the proof of the lemma says that 
no plane contains a subscheme of $P^{\perp}$ of length $\geq 5$,
so the cases $\ell (P^{\perp}_-)=5$, $6$, $7$ can't happen. 
If $\ell (P^{\perp}_-)=4$ then $P^{\perp}_-$ is contained in a line, and we can
choose a plane which meets furthermore a point of $P^{\perp}_+$, again
giving a plane with more than $5$ points. This shows that we must
have $\ell (P^{\perp}_-)\geq 8$ and since $\ell (P^{\perp}_+)\geq 3$ we get
$\ell (P^{\perp})\geq 11$ as claimed (under the hypothesis contrary to the lemma).
This contradicts the estimate of the previous lemma, which completes the
proof of the present one.  
\end{proof} 

Now choose a plane section $V$ such that $V\cap P^{\perp}$ has maximal length.
Write $P^{\perp}_+=P^{\perp}\cap V$ and let $P^{\perp} _-$ be the residual subscheme with respect to $V$.
Then $P^{\perp}_-$ satisfies $CB(2)$. If $\ell (P^{\perp}_+)=5$ then $\ell (P^{\perp}_-)\leq 5$
and by \cite{MestranoSimpson}, $P^{\perp}_-$ must consist of $4$ or $5$ points on a line. 
Choose a new plane section passing through this line but not meeting $P^{\perp}_+$;
we conclude that $P^{\perp}_+$ must also consist of $5$ points on a line, but
then in fact we could choose a plane section meeting $P^{\perp}$ in $6$ points.
Thus the case $\ell (P^{\perp}_+)=5$ doesn't happen. 
If $\ell (P^{\perp}_+)\geq 7$ then $P^{\perp}_-$ would consist of $\leq 3$  points,
but there are no such subschemes satisfying $CB(2)$, so this can't happen
either. We conclude that $\ell (P^{\perp}_+)=6$, hence $P^{\perp}_-$ must be $4$ points
on a line. If $y$ is any point of $P^{\perp}_+$ then there is a plane containing
$P^{\perp}_-$ and $y$, so the remaining points of $P^{\perp}_+$ are either
on this same plane, or else contained in a line. If the two lines meet at a
point, this would give a plane section containing too many points. 
Hence, we conclude that there are $2$ skew lines containing at least $8$ of
the $10$ points in $P^{\perp}$. Because of $CB(3)$ for $P^{\perp}$, in fact all of
the points must be on the two skew lines. 

Count now the dimension of the space of such configurations: there are $8$
parameters for the two skew lines. 
Once this configuration is fixed, the subscheme $P^{\perp}$ is specified up to a finite set of choices. 
The choice of $W$ counts for $3$, and the choice of $10$ points in $W$ counts for $10$.
The full dimension of this space of choices is therefore $\leq 21$. As in the
proof of the previous lemma, the case where our double point at $p'\in B_2$ lies
on $W$ but the tangent direction extends out of $W$, doesn't add an extra dimension because 
we get a point participating in $P^{\perp}$ which constrains the choice of points on $W$. 
In view of the fact that ${\rm dim}({\bf H}^{ns}_X)=24$, this situation cannot happen
for a general $E$. 

This contradiction
completes the proof of the proposition: for a general $E$ in its irreducible component, the base
locus $B_1$ has dimension $0$. 
\end{proof}

\section{The structure of a general zero scheme $P$}
\label{Pstructure}

The previous results were as close as we could get to saying that $E(1)$ is
generated by global sections, with the techniques we could find. Choose a general section
$s\in H^0(E(1))$ and let $P$ denote its scheme of zeros. If $y\in B_1$ (but not in $B_2$)
then a general section will not vanish at $y$. Furthermore, if there is a point $p'$ in $B_2$
then the structure of $P$ near $p'$ is at most an infinitesimal tangent vector.

\begin{proposition}
\label{structureofP}
Suppose $s\in H^0(E(1))$ is a general element, and let $P$ be the subscheme of zeros. 
We can write $P=P'\cup P''$ where $P'$ consists of the possible point of $P$ located
at $B_2$, and $P''$ is all the rest. With this notation, $P''$ consists of $18$, $19$ or
$20$ isolated points, and $P'$ is respectively an infinitesimal tangent vector at $p'\in B_2$;
or the isolated point $p'$; or empty. At any point $y\in P''$, the map
$$
H^0(J_P(3))\rightarrow J_y /J_y^2 (3)
$$
is surjective, meaning that $y$ is locally the complete intersection of two general
sections of $H^0(J_P(3))$. 
\end{proposition}
\begin{proof}
By Proposition \ref{B1dimzero}, the base locus $B_1$ has dimension zero. 
For any point $z\in B_1$ with $z\not \in B_2$, a sufficiently general section 
$s\in H^0(E(1))$ is nonzero at $z$. Therefore, for $s$ general the scheme of zeros $P$
doesn't meet $B_1$ except possibly at $B_2$. Divide $P$ into two pieces,
$P'$ at $B_2$ and $P''$ which doesn't meet either $B_1$ or $B_2$.
By Proposition \ref{B2onepoint}, the
base locus $B_2$ consists of at most one point which we shall denote by
$p'$ if it exists. Therefore, $P'$ is either empty or has a single point.
By Corollary \ref{lengthtwoatB2},
the zero scheme of a general $s$ at $p'$ has length at most two. 
So, if $P'$ has a point, then it is scheme-theoretically either this reduced point,
or an infinitesimal tangent vector there. 

At a point $y\in P''$, since $y$ is not in the base locus $B_1$, it means
that $E(1)$ is generated by global sections at $y$. From the
standard exact sequence \eqref{standardexact1} for $s$ we see that $E(1)_y = J_y /J_y^2 (3)$,
so the generation of $E(1)_y$ by global sections is exactly the
surjectivity of the last claim in the proposition. 
\end{proof}

The points of $P''$ are ``interchangeable''. This can be phrased using Galois theory.
Write $H^0(E(1))= \aaa ^5_{\cc}= {\rm Spec}\cc [t_1,\ldots , t_5]$. 
Put $K= \cc (t_1,\ldots , t_5)$ and let ${\bf s}\in \aaa ^5_K$ be the tautological point.
Think of ${\bf s}\in H^0(X_K, E(1))$. Let $P\subset X_K$ be the subscheme of zeros. 
The decomposition $P=P'\cup P''$ is canonical, hence defined over $K$. 
On the other hand, $P''$ consists of $18$ to $20$ points, but the points are only
distinguishable over $\overline{K}$, which is to say $P'' _{\overline{K}}\subset X(\overline{K})$
is a set with $18$, $19$ or $20$ points. The Galois group ${\rm Gal}(\overline{K}/K)$
acts.

\begin{proposition}
\label{doublytransitive}
The action of ${\rm Gal}(\overline{K}/K)$ on the set $P''_{\overline{K}}$ is
doubly transitive: it means that any pair of points can be mapped to any other pair.
\end{proposition}
\begin{proof}
For general $s$, the part $P''$ is contained in the open subset $X^g$ where $E(1)$
is generated by global sections. 
Suppose $x_0,y_0$ and $x_1,y_1$ are two pairs of points in $P''$. Consider a continuous path
of pairs $(x(t),y(t))$ contained in $X^g\times X^g$ defined for $t\in [0,1]\subset \rr$, 
with $(x(0),y(0))=(x_0,y_0)$ and
$(x(t),y(t))=(x_0,y_0)$. 
Vanishing of a section at $x(t)$ and
$y(t)$ imposes $4$ conditions on elements of the $5$-dimensional space
$H^0(E(1))$, so we get a family of sections $s(t)$ leading to a family of subschemes $P''(t)$.
For general choice of path, the $P''(t)$ will all be reduced with $18$, $19$ or $20$ points.
At $t=0$ and $t=1$, the section is the same as $s$ up to a scalar since it is uniquely
determined by the vanishing conditions. We obtain an element of the fundamental group of 
an open subset of the parameter space of sections $s$, whose action on the covering determined
by the points in $P''$, sends $(x_0,y_0)$ to $(x_1,y_1)$. This shows that the action is doubly
transitive, and it is the same as the Galois action after applying the Grothendieck correspondence
between Galois theory and covering spaces. 
\end{proof}

\begin{corollary}
\label{singlecomponent}
Suppose $P$ is the scheme of zeros of a general section $s\in H^0(E(1))$,
written $P=P'\cup P''$ as above. Let
$Z\subset \pp ^3$ be the intersection of two cubic hypersurfaces corresponding
to general elements of 
$H^0(J_P(3))$. This $Z$ is a complete intersection: ${\rm dim}(Z)=1$.
There is a single irreducible component $Z''$ of $Z$
such that $P''$ is contained in the smooth locus of $Z''$. At points of $P''$, $Z''$ is transverse to $X$. 
The only irreducible components of $Z$ which can be non-reduced are those, other than $Z''$,
which are fixed as the cubic hypersurfaces vary. 
\end{corollary}
\begin{proof}
The points of $P''$ lie in the subset $X^g$ where sections generate $E(1)$. 
In the standard exact sequence \eqref{standardexact1}, sections of $E(1)$ map to sections of $J_P(3)$,
and the fiber $E(1)_x$ maps to $J_x/J_x^2(3)$ for $x\in P''$. As sections generate
the fiber, it implies that sections of $J_P(3)$ generate the ideal $J_x$ (which is
the maximal ideal at $x$). Two general sections therefore have linearly independent
derivatives at $x\in P''$ when restricted to $X$, so the same is true of the cubics in $\pp ^3$
which means that $Z$ is a transverse complete intersection at any $x\in P''$. 
The doubly transitive action from the previous proposition implies that for general sections
and general choice of $Z$, the points of $P''$ must all lie in the same irreducible component
$Z''$ of $Z$. Note, on the other hand, that any other component of $Z$ must also have dimension $1$,
otherwise we would get a $1$-dimensional base locus $B_1$ of sections of $E(1)$ on $X$
and this possibility has been ruled out in Proposition \ref{B1dimzero}.

Note that $Z''$ is reduced since its smooth locus is nonempty. If $Z_i$ is a non-reduced component,
then by Sard's theorem it has to be a fixed part of the family of complete intersections of the form $Z$.
\end{proof}

\section{Complete intersections of two cubics}
\label{completeintersections}

We need to know something about what curves $Z$ can arise as the
complete intersection of two cubics in $\pp ^3$. The degree of $Z$ is $9$.
If $Z$ is smooth,
then $K_Z= \Oo _Z(2)$ is a line bundle of degree $18$, so the genus of $Z$ is $10$. 

The choice of $Z$ corresponds to a choice of two-dimensional subspace $U\subset H^0(\Oo _{\pp ^3}(3))=\cc ^{20}$;
furthermore $U=H^0(J_Z(3))$ and $h^0(\Oo _Z(3))= 18$ (as may be seen from the exact sequences
of restriction to one of the cubics $C$ and then from $C$ to $Z$).
The dimension of the Grassmanian of $2$-dimensional subspaces of $\cc ^{20}$ is $2\cdot (20-2)=36$. 
Denote this Grassmanian by ${\bf G}$; we have a universal family
$$
{\bf Z} \subset {\bf G}\times \pp ^3.
$$
Let ${\bf G}_{rci}$ denote the subset of $U$ defining a reduced complete intersection, i.e.\  such that the
fiber $Z_U$ of ${\bf Z}$ over $U$ has
dimension $1$ and is reduced. 

Suppose $Z=Z'\cup Z''$ is a decomposition with $d':= {\rm deg}(Z')$, $d'':={\rm deg}(Z'')$,
so $d"'+d'' = 9$ and we may assume $d'\leq d''$. We aren't saying necessarily that the pieces
$Z'$ and $Z''$ are irreducible, though.  
This gives $d'\leq 4$.

\begin{lemma}
\label{leq5inconic}
Suppose $Z$ is a complete intersection of two cubic hypersurfaces in $\pp ^3$. If $Z_i$ is
a reduced irreducible component of $Z$ of degree $\leq 5$,
then either $Z_i$ is contained in a quadric,
or the normalization of $Z_i$ has genus $g=0$ or $1$ and the space of such curves has dimension $\leq 20$. 
\end{lemma}
\begin{proof}
If ${\rm deg}(Z_i)\leq 4$ then it is contained in a quadric, so we may assume the degree is $5$. 
Let $Y\rightarrow Z_i$ be the normalization and let $g$ denote the genus of $Y$. 
Projecting from a point on $Z_i$ gives a presentation of $Y$
as the normalization of a plane curve of degree $4$, so it has genus $g\leq 3$. The line bundle $\Oo _Y(2)$
has degree $10$ which is therefore in the range $\geq 2g-1$, so $h^0(\Oo  _Y(2))=11-g$. 
If $g\geq 2$ then this is $\leq 9$ and the map $H^0(\Oo _{\pp ^3}(2))\rightarrow H^0(\Oo _Y(2))$
is not injective, giving a quadric containing $Z_i$. 

If $g=0$, $Y\cong \pp ^1$,
and the embedding to $\pp ^3$ corresponds to a map $\cc ^4\rightarrow H^0(\Oo _Y(1))\cong \cc ^6$.
This yields $24$ parameters, minus $1$ for scalars, minus $3$ for $Aut(\pp ^1)$, so there are
$20$ parameters. 

If $g=1$, the moduli space of elliptic curves provided with a line bundle $\Oo _Y(1)$ of degree $5$,
has dimension $1$ (the line bundles are all equivalent via translations). 
Here $h^0(\Oo _Y(1))=\cc ^5$ so the space of parameters for the embedding has dimension $19$,
this gives a $20$ dimensional space altogether. 
\end{proof}

\begin{lemma}
\label{63}
Let ${\bf G}_{rci}(6,3)$ denote the locally closed subset of ${\bf G}_{rci}$ parametrizing 
complete intersections $Z_U$ such that $Z_U = Z'\cup Z''$ with $Z''$ irreducible of degree $6$.
The degree $3$ piece $Z'$ is allowed to have other irreducible components. Then 
${\bf G}_{rci}(6,3)$ is the union of four irreducible components parametrizing:
\newline
(a)\, the case where $Z'$ is a rational normal space cubic;
\newline
(b)\, the case where $Z'$ is a plane cubic; 
\newline
(c)\, the case where $Z'$ is a disjoint union of a plane conic and a line; and
\newline
(d)\, the case where $Z'$ is a disjoint union of three lines.

These have dimensions $28$, $30$, $26$ and $24$ respectively. For (a) there is an open
set on which $Z_U=Z' \cup Z''$ with $Z'$ and $Z''$ being smooth and meeting transversally
in $8$ points. 
\end{lemma}
\begin{proof}
We divide into cases corresponding to the piece $Z'$ obtained by removing the degree $6$
irreducible component $Z''$. 
In the first case (a), we include all degree $3$ curves $Z'$ which are connected chains of rational curves
with no loops or self-intersections, which then have to span $\pp^3$. In these cases, $H^0(\Oo_{Z'} (1))$
is always $4$-dimensional, and $Z'$ deforms to a smooth rational normal space cubic. 
For a given $Z'$, the space of choices of $U$ is the Grassmanian of $2$-planes in $H^0(J_{Z'}(3))$
and $H^0(\Oo _{Z'} (3))$ has  dimension $10$; one can check (case by case) 
that the restriction map from 
$H^0(\Oo _{\pp^3}(3))$ is surjective, so $h^0(J_{Z'}(3))=10$ and the Grassmanian has dimension $16$.
The space of choices of rational normal space cubic is irreducible, equal to the space of choices of basis for
the $4$-dimensional space $H^0(\Oo _{Z'}(1))$ ($16d$), modulo scalars ($1d$)
and the automorphisms of the rational curve $Z'$ ($3d$). 
In the case of a chain the dimension of the automorphism group goes up so those pieces are of 
smaller dimension in the closure of the open set where $Z'$ is smooth. The dimension of this component
is therefore
$$
{\rm dim}{\bf G}_{rci}(6,3)^{(a)} = 16 + 16 - 1 - 3 = 28.
$$
A general point corresponds to a smooth $Z'$ with general choice of $U$ yielding a smooth curve $Z''$ of
degree $6$ meeting $Z'$ at $8$ points. 

The remaining possibilities are (b), (c) and (d), which are irreducible
and one counts the dimensions as:
\newline
(b)\, a plane $3d$ plus a cubic $9d$ plus a subspace of $H^0(J_{Z'}(3))=\cc ^{11}$, $18d$ for a total of $30$; 
\newline
(c)\, a plane $3d$ plus a conic $5d$ plus a disjoint line $4d$ plus a subspace of $H^0(J_{Z'}(3))=\cc ^9$, $14d$ for
a total of $26$; 
\newline
(d)\, three lines $12d$ plus a subspace of $H^0(J_{Z'}(3))=\cc ^8$, $12d$ for a total of $24$. 
\end{proof}

\begin{lemma}
\label{72}
Let ${\bf G}_{rci}(7,2)$ denote the locally closed subset of ${\bf G}_{rci}$ parametrizing 
complete intersections $Z_U$ such that $Z_U = Z'\cup Z''$ with $Z''$ irreducible of degree $7$.
The degree $2$ piece $Z'$ is allowed to have other irreducible components. Then 
${\bf G}_{rci}(7,2)$ is the union of two irreducible components parametrizing:
\newline
(a)\, the case where $Z'$ is a plane conic; and
\newline
(b)\, the case where $Z'$ is a disjoint union of two lines.

These have dimensions $30$ and $28$ respectively. For (a) there is an open
set on which $Z_U=Z' \cup Z''$ with $Z'$ and $Z''$ being smooth and meeting transversally
in $6$ points. 
\end{lemma}
\begin{proof}
The complementary curve $Z'$ has degree $2$. If irreducible, it has to be a plane conic.
If reducible, it is the union of two lines. If the lines meet, this still corresponds to (a),
if they are disjoint it is case (b). Both cases have irreducible spaces of parameters. 

To count the dimensions, in case (a) the choice of plane $H\cong \pp ^2\subset \pp ^3$ is
$3d$, the choice of conic in the plane is $5d$, and $h^0(\Oo _{Z'}(3))= 7$.
One can check that the restriction map is surjective (since $Z'$ is reduced there are only two cases, an
irreducible conic or two crossing lines); so $h^0(J_{Z'}(3))= 13$ and the Grassmanian of $2$-planes
in here has dimension $22$. The total dimension is therefore $3+5+22=30$. 

In case (b) the choice of two disjoint lines is $8$ dimensional, and $h^0(J_{Z'}(3))= 12$;
the Grassmanian of $2$-planes has dimension $20$ so the total dimension here is $28$. 
\end{proof}

\begin{lemma}
\label{81}
The subvariety ${\bf G}_{rci}(8,1)$ parametrizing $Z_U = Z'\cup Z''$ with $Z''$ irreducible of degree $8$,
is irreducible of dimension $32$. It has an open
set on which $Z''$ is smooth and meets the line $Z'$ transversally in $4$ points. 
\end{lemma}
\begin{proof}
Note that $Z'$ has to be a line. The space of lines has dimension $4$ and $h^0(J_{Z'}(3))=16$.
The Grassmanian of $2$-planes $U$ in $H^0(J_{Z'}(3))$ has dimension $28$, so the total dimension is $32$. 
The general element is contained in a smooth cubic surface, on which the relevant linear system defining
$Z''$ has no base points so a general $Z''$ is smooth; 
the intersection $Z'\cap Z''$ has $4$ points by the adjunction formula. 
\end{proof}

\section{The common curve case}
\label{sec-ccc}

Consider a general bundle $E$ in its irreducible component, a general section $s\in H^0(E(1))$,
and a general two-dimensional subspace $U\subset W:= H^0(J_P(3))$. Let $Z\subset \pp ^3$ be the
intersection defined by $U$, which is a complete intersection by Corollary \ref{singlecomponent}. 
Write $P=P'\cup P''$
as usual and let $Z''\subset Z$ be the irreducible component containing $P''$. 

Let $Q\subset \pp ^3$ be the intersection of the four independent cubics spanning $W=H^0(J_P(3))$.
It is contained in $Z$,  indeed it is the intersection of $Z$ with the other two cubics spanning
the complement of $U\subset W$. Hence ${\rm dim}(Q)\leq 1$, and also of course $P\subset Q$.
Write $Q=Q_0\cup Q_1$ where $Q_1$ is the union of $1$-dimensional
pieces of $Q$ and $Q_0$ is the remaining $0$-dimensional part.
Notice that $Q_1\cap P$ and $Q_0 \cap P$ correspond to Galois invariant pieces in the 
situation where $s$ is a generic geometric point, so by Proposition \ref{doublytransitive}
it follows that if $P''\cap Q_1$ is nonempty then $P''\subset Q_1$ and similarly for $Q_0$.

Our situation therefore breaks down into two distinct cases:
\newline
--the {\em common curve case} when the $1$-dimensional part $Q_1$
contains the big variable part $P''$; 
or
\newline
---the {\em variable curve case} when $P''\subset Q_0$.

In this section, we would like to rule out the first possibility; reasoning by
contradiction suppose on the contrary that we are in the common curve case.
Since $Q_1\subset Z$, and by Corollary \ref{singlecomponent} there is a single
irreducible component $Z''$ of $Z$ containing $P''$, it follows that $Z''\subset Q_1$.
The common curve case is therefore equivalent to the following hypothesis,
which will be in vigour throughout the section until it is ruled out. 

\begin{hypothesis}
\label{ccc}
All of the sections in $W=H^0(J_P(3))$ vanish along $Z''$.
\end{hypothesis}

\begin{lemma}
\label{containedquartic}
Suppose that ${\rm deg}(Z'')\neq 6$. Then $Z''$ is contained in a quadric,
from which it follows that $P$ is contained in a quadric. 
\end{lemma}
\begin{proof}
If we can show that $Z''$ is contained in a quadric, then it follows
that $P$ is contained in the same quadric by Lemma \ref{inquadric}.
Suppose ${\rm deg}(Z'')\leq 5$. Then by Lemma \ref{leq5inconic}, either $Z''$ is contained in a quadric
or it runs in a space of dimension $\leq 20$. In the latter case, for each choice of $Z''$
we have a space of possible choices of $P$ of dimension $20$, $19+3=22$, or $18+5=23$
depending on whether $P'$ is empty, a single point, or an infinitesimal tangent vector. 
In all cases, this results in a space of possible subschemes $P$ of dimension $\leq 43<44$,
so by Proposition \ref{dimensions}
it can't contribute to a general point in the irreducible component. 

Suppose ${\rm deg}(Z'')=d\geq 7$ (and of course $Z''\subset Z$ so $d\leq 9$). 
Choose a hyperplane $H\cong \pp ^2\subset \pp ^3$
not passing through a point of $P''$,
and let $A:= H\cap Z''\subset \pp ^2$ be the intersection. It is finite of length $d$.
Using Hypothesis \ref{ccc}, we get a $4$-dimensional space $W$ of sections of 
$H^0(J_{Z''}(3))$ (one can note that, by the same argument as Lemma \ref{inquadric},
sections of $\Oo (3)$ vanishing on $Z''$ vanish also on $P$ so 
$H^0(J_{Z''}(3))=H^0(J_{P}(3))$). Consider the restriction map
$$
r:W\rightarrow H^0(J_{\pp ^2, A}(3)).
$$
If $w$ lies in the kernel, it means that $w$ factors as the linear form defining $H$
times a quadric. Then this quadric contains $Z''$, so it contains $P$ by Lemma \ref{inquadric}.

We now show that $r$ is not injective. 

\begin{lemma}
\label{cubicconditions}
Suppose $A\subset \pp ^2$ is a subscheme of length $7$. If $A$ doesn't contain $5$
points on a line (i.e.\  the intersection with any line has length $\leq 4$) then
$A$ imposes $7$ independent conditions on $H^0(\Oo _{\pp ^2}(3))$.
\end{lemma}
\begin{proof}
Choose a line $L$ with maximal value of the length $\ell$ of $L\cap A$.
Then $2\leq \ell \leq 4$. On the $10$ dimensional
space of cubics we can try to impose $3$ further conditions and should prove that this makes
sections vanish.
Imposing $0$,$1$ or $2$ additional conditions on 
cubics restricted to $L$ makes them vanish. The residual subscheme $A'$ of $A$
with respect to  $L$ has length $7-\ell$, and we have to show that it imposes
this number of conditions on conics. Again choosing a line $L'$
with maximal contact (of order $2\leq \ell ' \leq 3$) with $A'$, imposing $0$ or $1$ additional conditions
we get vanishing of the conics on $L'$; we are left with 
a further residual subscheme $A''$ of length $7-\ell -\ell '$, which is between $0$ and $3$;
however if $A''$ consisted of $3$ colinear points that would imply $\ell \geq 3$
so $A''$ would have length $\leq 2$ and this is ruled out. Hence $A''$ imposes
independent conditions on linear sections. We conclude that $A$ imposed $7$ independent conditions
on cubics. 
\end{proof}

To finish the proof of Lemma \ref{containedquartic}, consider the subscheme $A$ from that proof.
It has length $7$, $8$ or $9$. Since $H$ was general, no $5$ points of 
$A$ lie on a line. Applying the previous lemma to a subscheme of length $7$, we see
that $A$ has to impose at least $7$ conditions on cubics. But $r(W)$ is a subspace
of the $10$-dimensional $H^0(\Oo _{\pp ^2}(3))$, vanishing on $A$. Thus ${\rm dim}(r(W))\leq 3$
showing that $r$ can't be injective. This completes the proof. 
\end{proof}

To finish this section, we just have to consider the case when ${\rm deg}(Z'')=6$. 
Consider first the case where a general $Z$ is reduced, and apply Lemma \ref{63}.
Notice that for each choice of $Z$ a reduced complete intersection $Z=Z_U$, $U\in {\bf G}_{rci}(6,3)$,
the space of possible choices of $P$ has dimension $\leq 20$.
From $P\subset Z$ this is clear when $P$ is reduced. The other possibility is that $P'$ is an
infinitesimal tangent vector. In that case, $P''$ is to be chosen in the smooth subset of $Z''$,
giving an $18$ dimensional space of choices. When $P'$ is in the smooth part of $Z'$ it 
really only corresponds to a $1$ dimensional space of choices, giving $19$ in all; when 
$P'$ is in the singular set of $Z$, the choice of $p'$ is $0$-dimensional and the
choice of tangent vector $\leq 2$-dimensional, so we get a space of choices of dimension $\leq 20$ in
all. 

From Lemma \ref{63}, the space of choices of pairs $(P,Z)$ in case (a) has dimension $\leq 28+20=48$. 
This is the same as the dimension of the component of the Hilbert scheme we are looking at.
However, a general pair $(Z,P)$ with $Z=Z_U$ in the $28$ dimensional piece 
${\bf G}_{rci}(6,3)^{(a)}$ and $P\subset Z$ general, doesn't occur. Indeed,
the degree $6$ piece $Z''$ is a smooth curve of genus $3$ so
there are $22$ sections of $\Oo (4)$ on $Z''$, and imposing up to $20$ conditions
can't make the sections vanish there; on the other hand, we could start by imposing up to $2$ 
independent conditions
from $P'$ on the degree $3$ piece $Z'$.
Thus, a general choice of $P\subset Z$ imposes $20$ conditions on
the $27$-dimensional space $H^0(\Oo _Z(4))$, 
leaving only $7$ sections to add to $h^0(J_Z(4))=8$ giving $15$. So, for a general
choice of $P\subset Z$ we have $h^0(J_P(4))=15$ and $P$ can't satisfy $CB(4)$. 
Hence the space of $(P,Z)$
such that $Z$ decomposes with a degree $6$ piece $Z''$, is a proper subspace of our irreducible
component so for general bundles $E$ this case doesn't occur.

In case (b) of Lemma \ref{63}, consider the plane $H$ containing $Z'$. The subspace $U$
of cubics vanishing on $Z$ has dimension $2$, whereas $H^0(J_{H,Z'}(3))$ has
dimension $1$ (the plane cubic $Z'$ determines its equation uniquely up to a scalar).
Therefore the restriction map from $U$ to $H^0(\Oo _H(3))$ is not injective; but
an element $u\in U$ mapping to zero on $H$ must be a product of a quadric and the
linear equation of $H$. This gives a quadric containing $Z''$ and hence $P$. So, 
for bundles with $h^0(E)=0$, this case doesn't occur. 

In cases (c) and (d) of Lemma \ref{63}, the total dimension is $\leq 26+20$ which
is too small, so these don't contribute for general bundles $E$.

This completes the analysis of the case where a general $Z$ is reduced. If a general $Z$ is non-reduced,
the non-reduced components $Z_i$ must be fixed, but different from $Z''$. As $Z''$ is also fixed
(when we vary the two-dimensional subspace $U$ of cubics), there must be at least one variable component
$Z_j$. The degree of the complementary piece to $Z''$ is $3$ so the only possibility is 
a fixed line of multiplicity two, and a variable line of multiplicity $1$. But then we have a $4$ dimensional
space of cubics passing through the degree $8$ curve $Z'' \cup Z_i$, so as in 
the proof of Lemma \ref{containedquartic}, this would give a quartic containing $Z''\cup Z_i$. 

We have finished the proof of the following theorem ruling out the common curve case. 

\begin{theorem}
\label{variablecurve}
Hypothesis \ref{ccc} leads to a contradiction. Therefore, 
for a general seminatural bundle $E$ in its irreducible component
and a general section $s\in H^0(E(1))$ defining a scheme of zeros $P$,
if the intersection of the four cubics passing through $P$ has a $1$-dimensional
piece $Q_1$, then the big interchangeable collection of points $P''\subset P$
doesn't meet $Q_1$. In other words, we are in the variable curve case.
\end{theorem}

\section{The reducible variable curve case}
\label{sec-rvcc}

The common curve case is ruled out by the previous section. Hence we are in the variable
curve case, when $P'' \subset Q_0$. It means that the choice of $W$, which determines $Q$,
then determines $P''$ and hence almost $P'$ (note however that $P'$ could still be in 
a $1$-dimensional piece of $Q_1$). Let $U\subset W$ be a general $2$-dimensional subspace
determining a complete intersection $Z=Z_U$. In this section, we consider the case when
$Z$ is not irreducible, a possibility which we would like to rule out. 

As was argued before, the points of
$P''$ are indistinguishable under the Galois group; the subspace $U$ may be chosen defined
over the same field as $P$, so $P''$ must be contained in the smooth points of a single
irreducible component $Z''$ of $Z$. Write $Z=Z'\cup Z''$ where
the remaining piece $Z'$ is allowed to be reducible. 

Applying Lemmas \ref{leq5inconic} (as in the first paragraph
of the proof of \ref{containedquartic}) and \ref{inquadric} as 
well as the hypothesis $h^0(E)=0$ so $P$ is not contained in a quadric,
gives that ${\rm deg}(Z'')\geq 6$.

The idea is to use a dimension count. The dimensions of the cases go all the way up to
${\rm dim}{\bf G}_{rci}(8,1) = 32$. However, the subspace $W$ determines $P''$,
and in turn $W$ is determined by a smaller subset of points than $P$, so the dimension count
can still work. 

Choose a subscheme $P_{16}\subset P$ of length $16$ as follows: start with $P_2$ of length $2$ containing
$P'$. Note that $P_2$ imposes $2$ independent
conditions on $H^0(\Oo _{\pp ^3}(3))$.
Then for $3\leq i\leq 16$ let $P_i:=P_{i-1}\cup \{ p_i\}$ with $p_i$ chosen in $P''$
such that it imposes a nontrivial condition on $H^0(J_{P_{i-1}}(3))$. This exists
because 
$$
h^0(J_P(3))=4< 20-(i-1)=h^0(J_{P_{i-1}}(3)).
$$
For $i=16$ we get $P_{16}$ imposing $16$ independent conditions, and $P'\subset P_{16}$.
It follows that 
$$
W = H^0(J_P(3))= H^0(J_{P_{16}}(3)).
$$
In particular, $W$ is determined by $P_{16}$. However, the remaining four points of $P-P_{16}$
are all in $P''$, in particular they are reduced points. Because of the ``variable curve case''
Theorem \ref{variablecurve}, the intersection $Q$ of the cubics in  $W$ has dimension $0$ 
at the points of $P''$; therefore, the locations of the remaining four points are determined
(up to a finite choice) by $W$. We get that $P$ is determined by $P_{16}$. 

We may now count the dimension of the space of choices of pair $(P,Z)$ where $Z=Z_U$
for a general subspace $U\subset W$. The space of choices of $Z$ containing
a degree $6$ or degree $7$ piece, is $\leq 30$. The dimension of the space of choices of
$P_{16}$ inside $Z$ is $\leq 16$ if we assume $Z$ reduced, or $\leq 17$ in any case, so the total dimension there
is $\leq 47$ which is too small. For the case of $Z$ containing a piece $Z''$ of degree $8$,
we get a dimension of $32+16=48$ so this looks possible. However, the general element $Z$
of the parameter space corresponds to the union of a smooth degree $8$ curve $Z''$
meeting a line $Z'$ in $4$ points. In order to get to dimension $48$, we must have
$P$ general, in particular $P''$ is a general collection of $18$ points in $Z''$. 
Now $Z''$ has genus $7$. The line bundle $\Oo _{Z'}(3)(-P'')$ is a general one of degree $6$,
which on a curve of genus $7$ will not have any sections. Hence, all cubics containing $P$
must vanish on $Z''$, which would put us back into the ``constant curve case''. So, this
case doesn't occur. 

We have finished ruling out the possibility that $Z$ would be reducible, resulting in the following
theorem.

\begin{theorem}
\label{Zirred}
For a general seminatural bundle $E$ in its irreducible component
and a general section $s\in H^0(E(1))$ defining a scheme of zeros $P$,
choose a general $2$-dimensional subspace $U\subset W=H^0(J_P(3))$
defining a complete intersection $Z_U$. Then $Z_U$ is irreducible.
\end{theorem}

\section{Subschemes of an irreducible degree $9$ curve}
\label{deg9curve}

In this section we complete the proof that the Hilbert scheme ${\bf H}^{sn}_{\pp ^3}$  
is irreducible,
by treating the case $P\subset Z_U$ where $Z_U$ is an irreducible complete
intersection of degree $9$. 

We first indicate how to construct an open set of the Hilbert scheme. Consider a
smooth complete intersection curve $Z_U$ for a general $2$-dimensional
subspace $U\subset H^0(\Oo _{\pp ^3}(3))$. The Grassmanian of choices of $U$
has dimension $36$ and there is a dense open set where $Z=Z_U$ is smooth of
genus $10$. 

Now $P\subset Z$ will be a subscheme of length $20$, which is a Cartier divisor
since $Z$ is smooth.  By varying any collection of $10$ points, we obtain a
family which surjects to the Jacobian ${\rm Jac}^{20}(Z)$. The line bundle
$L=\Oo _Z(4)(-P)$ has degree $36-20=16$. Note that the map
$$
H^0(\Oo _{\pp ^3}(4))\rightarrow H^0(\Oo _Z(4))
$$
is surjective, with kernel of dimension $8$. Hence, in order to obtain 
$h^0(J_P(4))=16$ one should ask for $h^0(\Oo _Z(4)(-P))= 8$ that is,
$h^0(L)=8$. As $g=10$ we get $\chi (L)=16+1-10= 7$. The condition
$h^0(L)=8$ is therefore equivalent to $h^1(L)=1$ or by duality,
$h^0(K_Z\otimes L^{-1})= 1$. Now, $K_Z=\Oo _Z(2)$ has degree $18$,
so $M:= K_Z\otimes L^{-1}$ is a line bundle of degree $2$. Asking for it to
have a section is equivalent to asking that $M\cong \Oo _Z(x+y)$ for
a degree $2$ effective divisor $(x)+(y)\in Z^{(2)}\subset {\rm Jac}^2(X)$.
The dimension of choices of $M$ is $2$ and the space of choices
is irreducible. For each choice of $M$, 
we have $L:= K_Z\otimes M^{-1}$, and 
the space of choices of divisor $P$ such that
$\Oo _Z(4)(-P)=L$ is a projective space of dimension 
$$
\# (P) - {\rm dim} ({\rm Jac}(Z))= 10.
$$
Putting these together, we get an irreducible $12$ dimensional space of 
choices of $P\subset Z$ such that $h^0(J_P(4))=16$. Including the
variation of $Z$ in a $36$ dimensional space, these fit together to
form an irreducible $48$ dimensional variety. 

If we replace $P$ by a subscheme $P_1\subset P$ of colength $1$ in the
above argument, then $M$ changes to $M_1=M(z)=\Oo _Z(x+y+z)$ where
$(z)=P-P_1$. As this is a general point of $Z$, we still have $h^0(M_1)=1$
giving the Cayley-Bacharach condition $CB(4)$ for $P$. Hence, there is  
a dense open subset of the $48$ dimensional variety parametrizing
pairs $(Z,P)$ where $P$ satisfies $CB(4)$. This is our irreducible component
of  ${\bf H}^{sn}_{\pp ^3}[2]$. 
Abstracting out the choice of $Z$ gives an irreducible $44$-dimensional
component of the Hilbert scheme ${\bf H}^{sn}_{\pp ^3}$. 

\begin{theorem}
\label{only}
The irreducible component constructed above is the only one in
${\bf H}^{sn}_{\pp ^3}$.
\end{theorem}
\begin{proof}
The argument above shows the basic idea. However, we need to do some
more work to treat the case when $Z$ is singular and specially the
possibility of a point or infinitesimal tangent vector in $P'$. 
The first step is to rule out this last possibility.

\begin{lemma}
A general $P$ in its irreducible component is reduced.
\end{lemma}
\begin{proof}
Given $P$ we can choose a quintic surface $X$ containing it,
and write $P=P'\cup P''$. We have $P''$ reduced and if
$P'$ is non-reduced, it consists of a single infinitesimal tangent
vector. Furthermore we may assume that $P$ is at a smooth point
of its Hilbert scheme. Choose a local smoothing infinitesimal deformation of $P'$;
we would like to extend that to a deformation of $P$ preserving the $CB(4)$
condition. As the Cayley-Bacharach property is open, it is equivalent to preserving
the property $h^0(J_P(4))=16$. One can check that the obstruction to finding a
deformation of $P''$ which, when added to the given deformation of $P'$,
preserves $h^0(J_P(4))$, would be the existence of a section 
$t\in H^0(J_P(4))$ such that $t$ vanishes to order $2$ at all the points of $P''$
in $\pp ^3$. 

Consider a complete intersection of cubics $Z$ containing $P$, and we may assume that
$P''$ lies on the smooth locus of $Z$. From the results of the previous sections,
we may assume that $Z$ is an irreducible curve of degree $9$. Hence $\Oo _Z(4)$
is a line bundle of degree $36$. Our section $t$ vanishes at $2P''\subset Z$,
but also at the points of $P'$. Together these are at least $38$ points,
so it follows that $t$ vanishes on $Z$. Let $C\subset \pp ^3$ be one of the
cubics defining $Z$. The residual of the scheme $2P''$ of multiplicity $2$ at $P''$,
intersected with $C$, consists of all the points of $P''$. The restriction $t|_C$,
divided by the other equation of $Z$,
corresponds to a linear section vanishing at these points; but the points of $P''$
are not all contained in a plane (indeed they are not even contained in a quadric),
so $t|_C=0$. Then $t$ divided by the equation of $C$ is a linear form again vanishing
on $P''$, so it is zero. Thus, $t=0$. 

This proves that the obstruction to 
lifting our smoothing deformation of $P'$ to a deformation of $P$, vanishes. 
Therefore, for a general point $P$ the piece $P'$ has to consist of at most
a single reduced point. This proves the lemma. 
\end{proof}

Suppose next that $P$ is a Cartier divisor on $Z$. This will always be the case
at points of $P''$ which are smooth points of $Z$, but it remains a possibility
that $P'$ is a non-movable point at a singularity of $Z$. We will deal with this
problem below, but for now in the interest of better explaining the argument,
assume that $L:= \Oo _Z(4)(-P)$ is a line bundle which we may think of as being a restriction
from a small analytic neighborhood of $Z$. 

Now $Z$ is a complete intersection, so duality still applies. This
can be seen, for example, by using  Serre duality on $\pp ^3$ 
and the equations for $Z$ which provide resolutions for $\Oo _Z$; the local
$Ext$ sheaves may be tensored with $L$ which exists on a neighborhood of $Z$.  
We get
$$
H^i(Z,L|_Z) \cong H^{1-i}(Z, L^{-1}\otimes\Oo _Z (2))^{\ast}.
$$
Applying this to $L= \Oo _Z(4)(-P)$, we get 
$$
h^1(L|_Z)= h^0(\Oo _Z(-2)(P)).
$$
On the other hand, $\chi (L) = 7$ and as before, $h^0(J_P(4))= h^0(L|_Z) + 8$,
so the condition $h^0(J_P(4))=16$ is equivalent to $h^1(L|_Z)=1$,
which in turn is equivalent to asking that the degree $2$ line bundle
$\Oo _Z(-2)(P)$ be effective.

The Picard scheme ${\rm Pic}^0(Z)$ is still a group scheme, hence smooth;
and its tangent space at the origin is $H^1(\Oo _Z)$. The exact sequences
for $Z\subset C\subset \pp ^3$ (where $C$ is one of the cubics cutting out $Z$)
give $H^1(\Oo _Z)\cong H^3(\Oo _{\pp ^3}(-6)) = H^0(\Oo _{\pp ^3}(2))^{\ast}$
which is $10$-dimensional. 
So the group scheme, as well as its torsors ${\rm Pic}^d(Z)$ are
$10$ dimensional. An infinitesimal argument with exact sequences also shows
that for $10$ general points in $Z$, the map from the product of their tangent
spaces to the Picard scheme is surjective. As $P$ consists of $20$ points,
and the Picard scheme has dimension $10$, at least $10$ points can move 
generally, keeping the same divisor $P$. 

The effective divisors form a two dimensional subscheme of ${\rm Pic}^2(Z)$.
Thus, at a general $P\subset Z$ satisfying $h^0(J_P(4))=16$, the
Hilbert scheme of such $P$ has dimension $12$. 
The locus of singular $Z$ has dimension $\leq 35$, so the pairs $(Z,P)$
with $Z$ singular lie on a subscheme of dimension $\leq 47$, and 
cannot therefore correspond to a general bundle $E$ in its irreducible component.
This finishes the proof of Theorem \ref{only} in the case where $P$
corresponds to a Cartier divisor.

Some further argument is needed for the general case. The reader may calculate
directly that the dimension of the space of $(Z,P)$ such that $Z$
is a nodal curve and $P$ contains a point $p'$ located at the node,
is $<48$ and doesn't contribute. This indicates that we don't get a new
irreducible component in this way. 

To give a more complete argument, consider $(Z,P)$ with $Z$ singular 
(but still reduced and irreducible) and
$P$ including a point $p'\in P'$ located at a singular point of $Z$.
Consider general hyperplanes $H\subset \pp ^3$ passing through $p'$, let
$K:= (H\cap Z)_{p'}$ (meaning the local piece of $H\cap Z$ at
$p'$) and let $P^+=P''\cup K$. This is now a Cartier divisor on 
$Z$ so the previous considerations apply. Let $\ell$ denote the length of
$K$. The condition that $Z$ is not contained in
a plane means that the general intersection $H\cap Z$ can't be concentrated
at a single point, on the other hand $p'$ is singular in $Z$,
so $2\leq \ell \leq 8$. The exact sequences for complete intersections
imply that $K$ imposes $\ell$ independent conditions on cubics.

Our point $p'$ is in the base locus $B_2$ for the bundle $E$, meaning that 
sections in $H^0(J_{X,P}(3))$ vanish to order $\geq 2$ at $p'$ in $X$. This is true for
any general quintic $X$ passing through $P$, so 
sections of $H^0(J_{\pp ^3,P}(3))$ vanish to order $\geq 2$ at $p'$ in $\pp ^3$. 
In particular, $Z$ contains the multiplicity two fat point at $p'$. In turn,
this implies that $K$ contains the multiplicity two fat point at $p'$ in $H$. 

We have $h^0(J_{P^+}(3))\geq 17-\ell$, which translates, using duality and 
calculating the Euler characteristic, into $h^0(\Oo _{Z}(-2)(P^+))\geq 1$. 
That is to say, $\Oo _X(-2)(P^+)$ is an effective line bundle of degree $\ell +1$
(the case $\ell = 1$ would correspond to the case treated previously). 
The dimension of the space
of choices of $P^+$ satisfying this effectivity condition, at general $P''$
in its linear system, is $\leq 11+\ell$. Note that since $P$ is reduced,
$P^+$ determines $P$. 

For a given $K\subset H$, the space of choices of $Z$ passing through $K$
is the Grassmanian of $2$-planes in $\cc ^{20-\ell}$, so it has dimension
$2(18-\ell ) = 36-2\ell$. We consider the space of choices of 
$(p',H,Z,P)$. The choices of $p'\in H$ form a $5$ dimensional space. 
Let $k$ denote the dimension of the space of choices of $K\subset H$
located at a given point $p'$. 
Then altogether, the space of choices of $(p',H,Z,P)$ has dimension
$$
\leq 5 + (11+\ell ) + (36-2\ell ) + k = 52 +k-\ell . 
$$
This should be compared with the dimension of the Hilbert scheme,
plus the number of choices of $H$ ($2$-dimensional) for each $P$, which is
to say $50$. 

The dimension count is now taken care of by noting that $K\subset H$
contains the fat point of multiplicity $2$ at $p'$ and this part is fixed
without parameters. The remaining parameters for the choice of $K$ 
therefore correspond to the length of the remaining subscheme,
which is to say $k\leq \ell -3$. 
This gives a count of $\leq 49$ for the space of $(p',H,Z,P)$
corresponding to the singular situation, which is $<50$ so it doesn't contribute
to the general points of the Hilbert scheme of $(Z,P)$. This completes the proof 
of Theorem \ref{only}. 
\end{proof}

\section{Bundles on the quintic}
\label{onquintic}

To complete the proof of Theorem \ref{main}, we should go back from the Hilbert scheme
of $CB(4)$ subschemes in $\pp ^3$, to the Hilbert scheme of $CB(4)$ subschemes of 
a general quintic $X$. Note first of all that we have looked above
at the Hilbert scheme ${\bf H}^{sn}_{\pp ^3}[2]$ of pairs $(Z,P)$. However, for a given $P$
the space of choices of $Z$ is just a Grassmanian of $2$-planes $U\subset W\cong \cc ^4$.
So, irreducibility of the $48$-dimensional Hilbert scheme  $\{ (Z,P)\}$ 
implies irreducibility
of the $44$-dimensional Hilbert scheme ${\bf H}^{sn}_{\pp ^3}$ 
which for brevity we denote just by $\{ P\}$ and so forth. Consider now the incidence variety 
of pairs $(P,X)$ such that $X$ is a smooth quintic hypersurface containing $P$.
The map $\{ (P,X)\} \rightarrow \{ P\}$ is a fibration in projective
spaces of dimension $35$, indeed by the seminatural condition $P$ imposes
$20$ conditions on the $56$ dimensional space $H^0(\Oo _{\pp ^3}(5))$
and we should also divide out by scalars. Thus, the incidence variety  
$\{ (P,X)\} $ is irreducible of dimension $79$. The space of quintics denoted 
$\{ X\}$ is an open subset of $\pp ^{55}$, and the Hilbert scheme of 
choices of $P$ for a given general $X$, is the $24$-dimensional fiber of the map
$$
\{ (P,X)\}\rightarrow \{ X\} .
$$
Up to now, we have shown that the source of this map is irreducible. An additional
argument is needed to show that the fibers are irreducible. We will use the same
argument as was used in \cite{MestranoSimpson}, which was pointed out to us by A. Hirschowitz. 

The idea is to say that there is a specially determined irreducible component of each fiber; then
this component is invariant under the Galois action of the Galois group of the function field
of the base, on the collection of irreducible components of the fiber. On the other hand,
irreducibility of the total space means that the Galois group acts transitively on the
set of irreducible components of the fiber, and together these imply that the fiber is
irreducible.

In order to isolate a special irreducible component, notice that the singular locus
of the moduli space of bundles was identified in \cite{MestranoSimpson}. 
It has a some explicit
irreducible components corresponding to the choice of $CB(2)$ subschemes of length $10$ in $X$, yielding
the case of bundles with $H^0 (E)\neq 0$ (this is the case we have been explicitly
avoiding throughout the bulk of the argument above). We consider the $19$-dimensional component 
of the singular locus whose general
point is a bundle $E$ fitting into an exact sequence
$$
0\rightarrow \Oo  _X\rightarrow E \rightarrow J_R(1)\rightarrow 0
$$
where $R\subset Y$ is a general collection of $10$ points on $Y=X\cap C$ for a quadric $C$.

For a general such bundle $E$, there is a unique co-obstruction, which is to say
a unique exact sequence as above, and the Zariski normal space
to the singular locus may naturally be identified with $H^1(E)$ which has dimension $2$. 
The second order obstruction map is the same as the quadratic form associated to the
symmetric bilinear form obtained from duality $H^1(E)\cong H^1(E^{\ast}(1))^{\ast}=H^1(E^{\ast})$.
This quadratic form defines a pair of lines inside $H^1(E)$. These are the two actual normal directions
of the moduli space of bundles along the singular locus at $E$. In order to show that this component
of the singular locus meets a canonically defined irreducible component of the moduli space,
it suffices to show that these two lines are interchanged as $R$ moves about in the Hilbert scheme
of $10$-tuples of points in $Y$. 

The $2$-dimensional space $H^1(E)$ together with its quadratic form, depend only on
the
arrangement $R\subset \pp ^3$ of $10$ points on a quadric $C\cong \pp ^1\times \pp^1$, 
in a way we now explain. The homogeneous coordinates of the $10$ points give 
a map $\cc ^4\rightarrow \cc ^{10}$. 
We get a map ${\rm Sym}^2(\cc ^4)\rightarrow \cc ^{10}$, and the equation of the
quadric $C$ is an element of the kernel; as ${\rm Sym}^2(\cc ^4)$ has dimension $10$
itself, there is an element $\xi =(\xi _1,\ldots , \xi _{10})$ in the cokernel,
unique up to scalars. The $CB(2)$ condition, which holds for general $R$,
corresponds to asking that $\xi _i\neq 0$ for all $1\leq i\leq 10$. 
Therefore $\xi$ defines a nondegenerate symmetric bilinear form on $\cc ^{10}$ 
denoted 
$$
\langle X,Y\rangle := X \Delta (\xi )Y^t = \sum_{i=1}^{10} \xi _ix_iy_i 
$$
The condition that $\xi$ vanish on the image of ${\rm Sym}^2(\cc ^4)$
says that $\cc ^4\subset \cc ^{10}$ is an isotropic subspace. In other words,
it is contained in its orthogonal subspace $\cc ^4\subset (\cc ^4)^{\perp} \cong \cc ^6$.
The quotient $\cc ^6 / \cc ^4$ is our two-dimensional space $H^1(E)$ and 
$\Delta (\xi )$ induces a quadratic form on here. We are interested in the two isotropic
lines. Fix $9$ of the points in a general way; then our two dimensional subspace with quadratic
form, depends on a single choice of $r_{10}\in C$. A calculation shows that the discriminant
divisor of the quadratic form contains reduced components in $C$. So if one has a curve of points
$r_{10}\in Y$ which intersect this divisor transversally, the two lines are
interchanged when we go around the intersection point on the curve. 
Now, one can choose $X$ to pass through the given $r_1,\ldots , r_9$ as
well as transversally through a general reduced point on the discriminant divisor. 
For such $X$, the tangent directions are interchanged as $R$ moves around in $Y=X\cap C$,
so the same is also true for any general $X$. 

This completes the construction of a specified irreducible component of the moduli space of
bundles. Notice that for the singular points $E$ constructed above, we still
have $H^1(E(1))=0$, so a soon as we move off the singular locus to get $H^1(E)=0$,
this gives a bundle with seminatural cohomology. Thus, our specified irreducible component
corresponds to bundles with seminatural cohomology. Now, Hirschowitz's argument plus
Theorem \ref{only} saying that the Hilbert scheme of choices $\{ P\}$
is irreducible, combine to show that there is only one irreducible component 
in the moduli space of stable bundles on $X$ of degree $1$ and $c_2=10$ having seminatural cohomology. 
This completes the proof of Theorem \ref{main}. 

\section{Some ideas for the non-seminatural case}
\label{someideas}

We indicate here how one should be able to treat Conjecture \ref{all}.
Notice that we made the hypothesis that $H^1(E(1))=0$, and this implied seminatural
cohomology. So, in the non-seminatural case we have $h^1(E(1))\geq 1$,
and $h^0(E(1))\geq 6$. If $s:\Oo \rightarrow E(1)$ with subscheme of zeros $P$ then
$h^0(J_P(3))\geq 5$. 

The first step will be to show that sections of $E(1)$ have a base locus consisting
of at most one point $p'$, and that a general $P$ has to be reduced at $P'$,
with $19$ points making up $P''$ with doubly-transitive Galois action. This should be
similar to our arguments of Sections \ref{baseloci} and \ref{Pstructure}. 

One can also point out, right away, that this allows us to rule out the
``common curve case'' as in Section \ref{sec-ccc}, indeed even in the case when $Q_1$ has
degree $6$, the same argument as we used for degrees $7$ and $8$ works to show
that $Q_1$ would have to be contained in a quadric. 

So, we are in the variable curve case. If $Z$ is a complete intersection of two
cubics passing through $P$, then $Z=Z'\cup Z''$ with $Z''$ irreducible,
containing $P''$ in its smooth locus. Part of the argument consisted of ruling out
${\rm deg}(Z'')<9$ by a dimension count. Here we can't just transpose the arguments,
indeed the dimension of the Hilbert scheme of possible collections $P$ might
be strictly smaller than $44$, because each $P$ can contribute a positive dimensional
space of extension classes. 

So we should divide the argument into two cases. If $h^1(J_P(4))=1$, i.e.\  $e=0$ in 
the notations of Section \ref{notations}, then the dimension of ${\rm Ext}^1(J_P(2),\Oo _X(-1))$
is $1$ and the extension class is unique up to scalars. In this case, the
dimension of the Hilbert scheme $\{ P\}$ remains $44$ (and 
including the complete intersection curve $Z$ gives $\{ (Z,P)\}$ of dimension $48$).
The dimension count may then proceed as we have done and this should allow us to treat
this case.

In the case when $h^1(J_P(4))\geq 2$, each choice of $P$ corresponds to a
positive dimensional space of choices of extension class up to scalars. 
However, in this case we can degenerate the extension class to one which no longer
satisfies the Cayley-Bachrach condition---meaning that, viewed as a dual element
to $\Oo _P(4)$, it vanishes on one or more points. 

The doubly transitive Galois action on $P$ implies that the images of the points $P$ in
the projective space of extension classes, cannot generically bunch up in groups of more
than one. Therefore, it is possible to degenerate the extension class towards one
which vanishes at exactly one point of $P$. This means an
extension which corresponds to a torsion-free sheaf $E$ with a singularity at a single point.
It therefore corresponds to a point in the boundary of the moduli space, at the boundary
component coming from $M_X(2,1,9)$. This boundary piece has codimension $1$ and
we should be able to analyze the nearby bundles and conclude that we remain in
the principal irreducible component (indeed it suffices to say that nearby bundles
have seminatural cohomology). 

The technique of localizing our picture on the boundary of the moduli space is obviously
a  necessary and important one which needs to be further developed in order to 
treat this type of question. This will be left for a future work. 

Another interesting direction will be to look at Reider's theory of nonabelian Jacobians
\cite{Reider1} \cite{Reider2} for bundles on a general quintic surface. The structures we
have encountered in an {\em ad hoc} way in the course of our proof, are actually
pieces of Reider's theory.

\bibliographystyle{amsplain}

\end{document}